\documentclass[reqno,12pt]{amsart}
\usepackage{amssymb,amsthm,amsmath,amsfonts,amstext,mathrsfs,url}
\usepackage{latexsym}
\usepackage{epsfig}
\usepackage{fancyhdr}
\usepackage[unicode]{hyperref}

\newtheorem{thm}{Theorem}

\newtheorem{lem}{Lemma}
\newtheorem{prop}{Proposition}
\newtheorem{defin}{Definition}

\newtheorem*{prob14}{Problem 14}

\hypersetup{colorlinks=true,
linkcolor=blue,
citecolor=blue,
urlcolor=blue,
filecolor=blue,
bookmarksnumbered=true,
pdfstartview=FitH,
pdfhighlight=/N
}

\input xy
\xyoption{all}

\newcommand{\m}{\mathbf}

\newcommand{\bl}{\bullet}
\newcommand{\p}{\partial}
\def\d{\,{\rm{d}}}
\def\k{\mathbf{k}}
\def\v{\,{\varsigma}}

\title[Rational projective plane flows. II]
{The projective translation equation and rational plane flows. II. Corrections and additions}
\author[G. Alkauskas]{Giedrius Alkauskas}
\address{Vilnius University, Department of Mathematics and Informatics, Naugarduko 24, LT-03225 Vilnius, Lithuania}
\email{giedrius.alkauskas@mif.vu.lt}

\newcounter{exam}

	{\refstepcounter{noteno}%
	\begin{small}
	\medbreak\par\noindent{{\bf Note~\thenoteno}.}}%
	{\hfill{$\Box$}\end{small}\par\medbreak}

\newenvironment{Example}%
	{\refstepcounter{exam}%
	\medbreak\par\noindent{{\bf Example~\theexam}.}}%
	{\hfill{$\Box$}
	\par\medbreak}

\begin{document}
\begin{abstract}In this second part of the work, we correct the flaw which was left in the proof of the main Theorem in the first part. This affects only a small part of the text in this first part and two consecutive papers. Yet, some additional arguments are needed to claim the validity of the classification results.\\
\indent With these new results in a disposition, algebraic and rational flows can be much more easily and transparently classified. It also turns out that the notion of an algebraic projective flow is a very natural one. For example, we give an inductive (on dimension) method to build algebraic projective flows with rational vector fields, and ask whether these account for all such flows.\\
\indent Further, we expand on results concerning rational flows in dimension $2$. Previously we found all such flows symmetric with respect to a linear involution $i_{0}(x,y)=(y,x)$. Here we find all rational flows symmetric with respect to a non-linear $1$-homogeneous involution $i(x,y)=(\frac{y^2}{x},y)$. We also find all solenoidal rational flows. Up to linear conjugation, there appears to be exactly two non-trivial examples.
\end{abstract}

\date{November 29, 2016}
\thanks{The research of the author was supported by the Research Council of Lithuania grant No. MIP-072/2015.}

\maketitle

\section{Rational vector fields} 
\subsection{Introduction}
\label{intro}
Let $\m{x}=(x,y)$, which is denoted by $x\bl y$. The main Theorem in \cite{alkauskas} claims the following.
\begin{thm}
Let $\phi(x,y)=u(x,y)\bl v(x,y)$ be a pair of rational functions in $\mathbb{R}(x,y)$ which satisfies the functional equation 
\begin{eqnarray}
(1-z)\phi(\m{x})=\phi\Big{(}\phi(\m{x}z)\frac{1-z}{z}\Big{)},\quad z\in\mathbb{R},\label{funk}
\end{eqnarray}
and the boundary conditions
\begin{eqnarray}
\lim\limits_{z\rightarrow 0}\frac{u(xz,yz)}{z}=x,\quad 
\lim\limits_{z\rightarrow 0}\frac{v(xz,yz)}{z}=y.
\label{bound}
\end{eqnarray}
 Assume that $\phi(\m{x})\neq \phi_{{\rm id}}(\m{x}):=x\bl y$.
Then there exists an integer $N\geq 0$, which  is the basic invariant of the flow, called \emph{the level}. Such a flow $\phi(x,y)$ can be given by
\begin{eqnarray*}
\phi(x,y)=\ell^{-1}\circ\phi_{N}\circ\ell(x,y),
\end{eqnarray*}
where $\ell$ is a $1-$BIR ($1$-homogeneous birational plane transformation), and $\phi_{N}$ is the canonical solution of level $N$ given by
\begin{eqnarray}
\phi_{N}(x,y)=x(y+1)^{N-1}\bl \frac{y}{y+1}.\label{can}
\end{eqnarray}
\label{mthm}
\end{thm}
To prove this and all other classification Theorems  in \cite{alkauskas, alkauskas-un, alkauskas-ab}, many ingredients are needed. One of the steps is the algorithm for reducing rational vector fields. Though all the rest ingredients are mathematically correct, this one contains a flaw.\\

Therefore, we accentuate the following:
\begin{itemize}
\item[i)]In \cite{alkauskas}, Section 4, all steps II. through VI. are valid, except for the step I. which contains a logical flaw. 
\item[ii)] All the special examples of projective flows in \cite{alkauskas, alkauskas-un, alkauskas-ab}, which make the majority of the text, are valid. This includes rational, algebraic, unramified, abelian flows, and one non-abelian flow with a vector field $x^2+xy+y^2\bl xy+y^2$; another class of non-abelian flows was named \emph{pseudo-flows of level $0$} and was described in \cite{alkauskas}, Subsection 5.2. Open directions in the investigations of the projective translation equation, indicated in (\cite{alkauskas-un}, Sect. 5.1.) and which were expanded in \cite{alkauskas-comm, alkauskas-super1, alkauskas-super2, alkauskas-super3}, and in this paper, too, remain valid. Note also that the flow with the vector field $x^2-xy\bl y^2-2xy$, as given by (\cite{alkauskas-un}, Theorem 1) is indeed unramified. This vector field is given an extensive treatment in the Appendix of \cite{alkauskas-ab}.
\item[iii)] Four classification Theorems (Theorem in \cite{alkauskas}, Theorem 1 in \cite{alkauskas-un}, Theorem 2 and Theorem 3 in \cite{alkauskas-ab}) remain valid as stated, if we additionally assume that with a help of a $1$-BIR a rational vector field can be reduced to a pair of quadratic forms, or to a vector field when one of the components vanishes. As it turns out now, the last happens exactly for abelian flows of level $1$; see Definition \ref{defin-ab}.     
\item[iv)] After corrections and additions of the current paper, we can say the following. Classification of rational plane flows, that is, Theorem \ref{mthm}, remains correct as stated. The variety of algebraic, abelian and (possibly) unramified flows is bigger, if the vector field cannot be reduced to a pair of quadratic forms or to a vector field when one of the components vanishes.
\item[v)] With much more tools at hand, some steps in proving classification theorems can be significantly simplified. 
\item[vi)] The flaw was a fortunate one, since it forced to work on Step III. (\cite{alkauskas}, Subsection 4.3). See Subsections \ref{cor-rat} and \ref{cor-alg} of this paper where it is shown that with the current approach, steps II. through IV. in \cite{alkauskas}, as far as rational flows are concerned, are superfluous.  However, during work on Step III. we discovered an amazing vector field $x^2-2xy\bl y^2-2xy$. It gave rise to two wide areas of research, both new in the framework of differential geometry - the theory of projective superflows, and the theory of unramified projective flows. The first one was already significantly andvanced in case of linear groups $O(2)$, $O(3)$ and $U(2)$, see \cite{alkauskas-super1,alkauskas-super2, alkauskas-super3}. For example, there exists exactly $5$ superflows over $\mathbb{R}$ in dimension $3$: tetrahedral (symmetry group is of order $24$), octahedral ($24$), icosahedral ($60$), $4$-antiprismal ($16$), and $3$-prismal ($12$). The theory of unramified flows was begun in \cite{alkauskas-un}, and is under development in \cite{alkauskas-un2}.
\item[vii)]Finally, now the case of rational flows in dimension $3$ looks much more tractable and optimistic, if we manage to find and analogue of Theorem \ref{thm1} (see Subection \ref{fundament}) and prove everything by the same ``bootstrapping" method; see Subsection \ref{induct} for an inductive construction of algebraic or rational flows in any dimension.  
\end{itemize} 
\subsection{Transformation of the vector field}
Let $\varpi\bl\varrho$  be a pair of $2-$homogeneous rational functions. Let $P(x,y)$ and $Q(x,y)$ be two homogeneous polynomials of degree $d\geq 0$. Then
\begin{eqnarray}
\quad\ell_{P,Q}(x,y)=\frac{xP(x,y)}{Q(x,y)}\bl\frac{yP(x,y)}{Q(x,y)}
\label{bir}
\end{eqnarray}
is a $1$-homogeneous birational plane transformation ($1$-BIR for short) whose inverse is $\ell_{Q,P}$. If $\phi$ is a projective flow, so is $\ell^{-1}\circ\phi\circ\ell$. Also, if $T$ is a non-degenerate linear map, then $T^{-1}\circ\phi\circ T$ is a projective flow. The transformation of the vector field under these two basic transformations are given by (\cite{alkauskas}, Proposition 8), which we reproduce here. Let $v(\phi)$ stand for a vector field of a flow $\phi$. 
\begin{prop}
\label{prop1}
Suppose, $L$ is a non-degenerate linear map, and $\phi(\m{x})$ is a flow. Then
\begin{eqnarray*}
v\Big{(}L^{-1}\circ\phi\circ L;\m{x}\Big{)}=L^{-1}\circ v(\phi;\m{x})\circ L.
\end{eqnarray*}
Further, let a birational map $\ell_{P,Q}$ be given by (\ref{bir}), $A(x,y)=P(x,y)Q^{-1}(x,y)$, which is a $0$-homogeneous function. Suppose that $v(\phi,\m{x})=\varpi(x,y)\bl \varrho(x,y)$. Then
\begin{eqnarray}
v(\ell^{-1}_{P,Q}\circ\phi\circ\ell_{P,Q};\m{x})&=&\nonumber\\
\varpi'(x,y)\bl\varrho'(x,y)&=&\nonumber\\
A(x,y)\varpi(x,y)-A_{y}[x\varrho(x,y)-y\varpi(x,y)]&\bl&\label{vecconj}\\
A(x,y)\varrho(x,y)+A_{x}[x\varrho(x,y)-y\varpi(x,y)].&&\nonumber
\end{eqnarray}
As a corollary,
\begin{eqnarray*}
x\varrho'(x,y)-y\varpi'(x,y)=A(x,y)[x\varrho(x,y)-y\varpi(x,y)].
\end{eqnarray*} 
 \label{conjug}
\end{prop}
\begin{proof}In \cite{alkauskas}, a direct proof was given. This result was used essentially in \cite{alkauskas,alkauskas-un,alkauskas-ab, alkauskas-comm}. We can double-verify this crucial formula in the following alternative way. We know that for $z$ sufficiently small,
\begin{eqnarray*}
\frac{u(xz,yz)}{z}&=&x+\sum\limits_{i=2}^{\infty}z^{i-1}\varpi^{(i)}(x,y),\\
\frac{v(xz,yz)}{z}&=&y+\sum\limits_{i=2}^{\infty}z^{i-1}\varrho^{(i)}(x,y),
\end{eqnarray*}
where $\varpi^{(2)}=\varpi$, $\varrho^{(2)}=\varpi$, and $\varpi^{(i)}$, $\varrho^{(i)}$, $i\geq 3$ are given recurrently (see \cite{alkauskas}, Subsection 3.2). The above expansion and the recurrsion itself are essentially equivalent to the system (\ref{PDE-u}). Now, if $A=A(x,y)$, we have (recall that $A$ is $0$-homogeneous):
\begin{eqnarray*}
U&:=&u\big{(}xzA(x,y),yzA(x,y)\big{)}=xzA+A^{2}\varpi z^2+\text{higher terms},\\
V&:=&v\big{(}xzA(x,y),yzA(x,y)\big{)}=yzA+A^{2}\varrho z^2+\text{higher terms},\\
A(U,V)&=&A+A(A_{x}\varpi+A_{y}\varrho)z+\text{higher terms}.
\end{eqnarray*} 
So,
\begin{eqnarray*}
\frac{U}{zA(U,V)}\text{ (mod }z^2)=\frac{x+A\varpi z}{1+(A_{x}\varpi+A_{y}\varrho)z}\text{ (mod }z^2)=x+(A\varpi-xA_{x}\varpi-xA_{y}\varrho)z,\\
\frac{V}{zA(U,V)}\text{ (mod }z^2)=\frac{y+A\varrho z}{1+(A_{x}\varpi+A_{y}\varrho)z}\text{ (mod }z^2)=y+(A\varrho-yA_{x}\varpi-yA_{y}\varrho)z,
\end{eqnarray*}
and Euler's identity for a $0$-homogeneous function $A$, that is, $xA_{x}+yA_{y}=0$, gives again the desired identities. 
\end{proof}
Note that the orbits of the flow are given by $\mathscr{W}(x,y)=\mathrm{const}.$ Here $\mathscr{W}$ is a $1$-homogeneous function, a solution to
\begin{eqnarray}
\mathscr{W}(x,y)\varrho(x,y)+\mathscr{W}_{x}(x,y)[y\varpi(x,y)-x\varrho(x,y)]=0.\label{w-eq}
\end{eqnarray}
\begin{defin}
If there exist a positive integer $N$ such  that $\mathscr{W}^{N}(x,y)$ is a rational function, such a flow is called \emph{an abelian flow}, and such minimal $N$ is called its \emph{level}.
\label{defin-ab} 
\end{defin}
Rational and algebraic flows are always abelian flows; this is obvious, but it will follow from Theorem \ref{thm1}.
\subsection{Fundamental ODE}In particular, suppose that the transformation $\ell(x,y)=xA(x,y)\bl yA(x,y)$ (where now $A$ is an arbitrary $0$-homogeneous function, not necessarily rational) transforms the second coordinate of the vector field into $\varpi'(x,y)=y^2$. If we put $A(x,1)=f(x)$, $A(x,y)=f(\frac{x}{y})$, $\varpi(x)=\varpi(x,1)$, $\varrho(x)=\varrho(x,1)$, this, minding (\ref{vecconj}), gives the differential equation for $f$ as follows:
\begin{eqnarray}
f(x)\varrho(x)+f'(x)(x\varrho(x)-\varpi(x))=1.\label{f-eq}
\end{eqnarray}
So, the general solution to (\ref{f-eq}) is given by $f_{1}(x)+\frac{\sigma}{\mathscr{W}(x,1)}$, where $f_{1}$ is any special solution; see Section \ref{sec-2}. 
Note that this equation appeared in \cite{alkauskas} as the equation (33), only with a ``$-1$" instead of ``$+1$", which corresponds to the fact that the second coordinate of a vector field is transformed into $-y^2$, and the second coordinate of the flow itself is transformed into $\frac{y}{y+1}$. This choice of a sign is of course, inessential. The existence of rational solution to this ODE for rational projective flows was the consequence of reduction algorithm, which, as we will shortly explain, contains a serious flaw. However, while working on paper \cite{alkauskas-ab} it emerged that (\ref{f-eq}) plays a much deeper role in describing the flow itself. For example, Theorem \ref{thm1} immediately implies that for algebraic flow (and for its special case, rational flow), all solutions to this ODE are algebraic. Moreover, exploring deeper we get that there exists \emph{at least one} rational solution. We can be more precise: \emph{a posteriori}, for algebraic flows there exists exactly one rational solution to the ODE (\ref{f-eq}) if and only if the rational or algebraic flow is of level $N\geq 2$. I would like the reader to pay attention to (\cite{alkauskas}, Sect. 2.3, p. 286). One of the sentences reads as follows. 
\begin{itemize}
\item[$\bigstar$]{\it ``For level $N\geq 2$ flows, the corresponding differential equation} (that is, (\ref{f-eq})) {\it has a solution of the form $f_{1}(x)+\sigma f_{0}(x)$, where $f_{1}$ is rational while $f_{0}$ is algebraic but not rational".}
\end{itemize}
 We will soon see that this can be proved without appealing to reduction of vector fields, so all the results about rational flows, except those in (\cite{alkauskas}, Sect. 4.1), remain valid.\\
 
 If $x\varrho-y\varpi=0$, then $\varpi\bl\varrho=xJ(x,y)\bl y J(x,y)$, where $J(x,y)$ is a $1$-homogeneous rational function. This gives rise to a rational flow of level $0$. This was explored in (\cite{alkauskas}, Sect. 2.2). The vector field can be easily integrated, and it gives rational projective flow $\frac{x}{1-J(x,y)}\bl \frac{y}{1-J(x,y)}$. Henceforth we assume $x\varrho-y\varpi\neq 0$.\\
 
 If this is the case, we will need one fact. Let $u\bl v$ is any projective flow with a rational vector field. Suppose that for every pair $(x,y)$, the function $\frac{u(xz,yz)}{v(xz,yz)}$, as a function in $z$, is constant. Then
\begin{eqnarray*}
\frac{u(xz,yz)}{v(xz,yz)}=\gamma(x,y),\quad z\in\mathbb{R}.
\end{eqnarray*} 
Now, take the limit $z\rightarrow 0$. This gives $\gamma(x,y)=\frac{x}{y}$, $yu(xz,yz)=xv(xz,yz)$, and this implies $x\varrho-y\varpi=0$ - a contradiction. Thus,
\begin{eqnarray}
\text{For some pair }(x,y), \frac{u(xz,yz)}{v(xz,yz)}
\text{ is not a constant function in }z.
\label{not-constant}
\end{eqnarray}
 
\subsection{The flaw}The logical mistake is contained in (\cite{alkauskas}, Subsection 4.1 Step I, p. 301):
\begin{itemize}
\item[$\bigstar$]{\it ``In fact, we do not need to worry about the linear change $L<...>$. With this trick in mind, we can, without loss of generality, consider
\begin{eqnarray}
\varpi(x,y)=\frac{\alpha P(x,y)+\beta Q(x,y)}{D(x,y)}\quad \varrho(x,y)=\frac{\gamma P(x,y)+\delta Q(x,y)}{D(x,y)},
\label{perein2}
\end{eqnarray}
and solve the problem of finding $A$ with an additional $4$ variables $\alpha$, $\beta$, $\gamma$, and $\delta$ in our disposition, with a single crucial restriction $\alpha\delta-\beta\gamma\neq 0$".}
\end{itemize}

In fact, if we are more careful and worry about the linear change, we arrive at the observation that a linear change does not add an additional freedom in the reduction: the equations (47), (48) and (49) in \cite{alkauskas} are also subject to a linear changes. After contemplating some time, this becomes obvious by the following two reasons.\\

\textbf{1.} Let $\alpha,\beta,\gamma,\delta\in\mathbb{R}$, $\alpha\delta-\beta\gamma\neq 0$. As special examples show, the vector fields $\varpi\bl\varrho$ and $\widehat{\varpi}\bl\widehat{\varrho}=\alpha\varpi+\beta\varrho\bl\gamma\varpi+\delta\varrho$ are not in any way related. For example, take $\begin{pmatrix}\alpha&\beta\\ \gamma&\delta\\ \end{pmatrix}=\begin{pmatrix} 0 & 1\\1 & 0 \\ \end{pmatrix}$, and $\varpi\bl\varrho=x^2\bl y^2$. This is a level $1$ rational flow, explicitly given by
\begin{eqnarray*}
\phi(\m{x})=\frac{x}{1-x}\bl\frac{y}{1-y}.
\end{eqnarray*}
On the other hand, the flow with the vector field $y^2\bl x^2$ is a superflow over $\mathbb{C}$ \cite{alkauskas-super1}, linearly conjugate (over complex numbers) to the superflow over $\mathbb{R}$ with the vector field $x^2-2xy\bl y^2-2xy$; the latter was explored in detail in \cite{alkauskas-un}. In particular, this flow can be integrated in terms of Dixonian elliptic functions. Thus, flows with vector fields $x^2\bl y^2$ and $y^2\bl x^2$ are not in any way related.\\  

\textbf{2.} The logic was that suitably conjugating a flow linearly in advance, we can afterwards find such a $1$-BIR $\ell$ that a flow $\ell^{-1}\circ\phi\circ\ell$ has a vector field whose denominator is of lower degree than that of $\phi$. However, according to (Proposition 15, \cite{alkauskas}),
\begin{eqnarray*}
\ell_{P,Q}^{-1}\circ L^{-1}\circ\phi\circ L\circ\ell_{P,Q}=L^{-1}\circ\ell^{-1}_{P',Q'}\circ\phi\circ\ell_{P',Q'}\circ L,
\end{eqnarray*}
where homogenic polynomials $P'(x,y)=P\circ L^{-1}(x,y)$, $Q'(x,y)=Q\circ L^{-1}(x,y)$. Suppose now that the flow on the left has a vector field whose common denominator has a lower degree than the vector field of $\phi$. However, the flow on the left is equal to the flow on the right. Now, it is obvious that linear conjugation does not alter the degree of the denominator, so already the flow $\ell^{-1}_{P',Q'}\circ\phi\circ\ell_{P',Q'}$ has a vector field with the lower degree denominator. So, denominators can be lowered with $1$-BIR transformations $\ell_{P,Q}$ only, and linear conjugation should not be considered at all while carrying on a reduction. Thus, reduction of vector fields are possible, as numerous examples in (\cite{alkauskas}, Sect. 2) show, but in much more limited cases.

\section{The proof}
\label{sec-2}
\subsection{Fundamental system}
\label{fundament}
The correct missing ingredient of the proof follows from techniques developed in \cite{alkauskas-ab}. We know that if $\phi=u\bl v$ is a projective flow, then $(u,v)$ satisfies the system of (in fact, two independent) PDEs
\begin{eqnarray}
\left\{\begin{array}{c}
u_{x}(\varpi-x)+u_{y}(\varrho-y)=-u,\\
v_{x}(\varpi-x)+v_{y}(\varrho-y)=-v.
\end{array}\label{PDE-u}
\right.
\end{eqnarray} 
These and the boundary conditions (\ref{bound}) uniquely determine the flow. \\

On the other hand, as the consequence of results in \cite{alkauskas-ab}, we get the following fundamental claim, which immediately lets to integrate the vector field $\varpi\bl\varrho$ in closed form in many cases, and replaces the differential system into the system of (generally, transcendental) non-differential equations. 
\begin{thm}Suppose, $x\varrho-y\varpi\neq 0$, and a pair of functions $(u,v)$ satisfies the system of PDE (\ref{PDE-u}) with the boundary conditions (\ref{bound}). Then the following system is satisfied: 
\begin{eqnarray}
\left\{\begin{array}{c}
\mathscr{W}(u,v)=\mathscr{W}(x,y),\\
\frac{1}{v}f\Big{(}\frac{u}{v}\Big{)}=\frac{1}{y}f\Big{(}\frac{x}{y}\Big{)}-1.
\end{array}\label{impl}
\right.
\end{eqnarray}
Here $f(x)$ and $\mathscr{W}(x,y)$ are the solutions of, respectively, (\ref{f-eq}) and (\ref{w-eq}). 
\label{thm1}
\end{thm}
These, together with boundary conditions (\ref{bound}), yield the flow. As we will soon see, this Proposition is an example of ``bootstrapping" phenomenon. In fact, boundary conditions can also be accommodated into the system (\ref{impl}) as follows. Let $u^{z}(x,y)=z^{-1}u(xz,yz)$, $v^z=z^{-1}v(xz,yz)$. Then the second equation of (\ref{impl}) gives, after a substitution $(x,y)\mapsto (xz,yz)$,
\begin{eqnarray}
\frac{1}{v^{z}}f\Big{(}\frac{u^{z}}{v^{z}}\Big{)}=\frac{1}{y}f\Big{(}\frac{x}{y}\Big{)}-z.\label{z}
\end{eqnarray} 
So, we choose the solution $(u^{z},v^{z})$ of this and $\mathscr{W}(u^{z},v^{z})=\mathscr{W}(x,y)$ which is equal to $(x,y)$ at $z=0$. If fact, only when dealing with rational flows we have no question of ramification. In all other cases the second equation of the system (\ref{impl}) should be interpreted in the sense just described.\\
 
Though Theorem \ref{thm1} gives two equations, we can confine just to the second one due to the following fact. Note that for any integer $N$, the solution to the homogeneous differential equation
\begin{eqnarray*}
Ng(x)\varrho(x)+g'(x)(x\varrho(x)-\varpi(x))=0
\end{eqnarray*}
is given by $\mathscr{W}^{-N}(x,1)$.
So, if we take a look at the case $N=-1$, it implies that the general solution to (\ref{f-eq}) is given by $f(x)=f_{1}(x)+\frac{\sigma}{\mathscr{W}(x,1)}$, where $f_{1}$ is any special solution. Now, suppose that $f(x)$ satisfies the second equation of (\ref{impl}). Since $\sigma\in\mathbb{R}$ is arbitrary, this gives the second equation of (\ref{impl}) for $f_{1}(x)$, and the first equation of (\ref{impl}) for $\mathscr{W}$.  
 
\begin{proof} The formulas (\ref{impl}) were worked out in detail for special examples in (\cite{alkauskas-ab}, Sections 4.2, 6, 7, 8) and also in (\cite{alkauskas-super1}, Section 3), but they hold in general. Indeed, let us take a look at (\cite{alkauskas-ab}, p. 748, formula (24)). It claims that
\begin{eqnarray}
u(x,y)\bl v(x,y)=\frac{\k(a-\tilde{\v})\tilde{\v}}{a-\tilde{\v}}f\big{(}
\k(a-\tilde{\v})\big{)}\bl
\frac{\tilde{\v}}{a-\tilde{\v}}f\big{(}\k(a-\tilde{\v})\big{)},
\label{form}
\end{eqnarray}
where $\tilde{\v}=\mathscr W(x,y)$, $\alpha(x)=f(x)\mathscr{W}(x,1)$, $a=\alpha(\frac{x}{y})$, and $\mathbf{k}$ is the inverse of $\alpha$. 
Thus,
\begin{eqnarray*}
\frac{u}{v}=\k(a-\tilde{\v})\mathop{\Longrightarrow}^{(\ref{form})}
\frac{1}{v}f\Big{(}\frac{u}{v}\Big{)}=
\frac{a}{\tilde{\v}}-1.
\end{eqnarray*}
Finally, the last displayed line on (p. 748, \cite{alkauskas-ab}), or the above formulas, too, show that $\frac{a}{\tilde{\v}}-1=\frac{1}{y}f(\frac{x}{y})-1$. This gives (\ref{impl}).\\

 Note that the formulas (\ref{form}) do not depend which special solution $f(x)$ to (\ref{f-eq})  is being used.  Indeed, keeping $\tilde{\v}$ fixed, and taking different $f$ changes it by a summand $\frac{\sigma}{\mathscr{W}(x,1)}$. Let new values are denoted by adding $``\star"$:
\begin{eqnarray*}
f_{\star}=f+\frac{\sigma}{\mathscr{W}(x,1)},\quad
\alpha_{\star}=\alpha+\sigma,\quad\mathbf{k}_{\star}(x)=\mathbf{k}(x-\sigma),\quad a_{\star}=a+\sigma.
\end{eqnarray*}
Thus, $R:=\k(a-\tilde{\v})=\k_{\star}(a_{\star}-\tilde{\v})$ does not change, and so $\frac{u_{\star}}{v_{\star}}=\frac{u}{v}$ does not change. Further, we have:
\begin{eqnarray*}
v_{\star}=\frac{\tilde{\v}}{a_{\star}-\tilde{\v}}f_{\star}(R)&=&\frac{\tilde{\v}}{a+\sigma-\tilde{\v}}\Big{(}f(R)+\frac{\sigma}{\mathscr{W}(R,1)}\Big{)}
=\frac{\tilde{\v}}{a+\sigma-\tilde{\v}}\Big{(}f(R)+\frac{\sigma f(R)}{\alpha(R)}\Big{)}\\
&=&\frac{\tilde{\v}f(R)}{a+\sigma-\tilde{\v}}\Big{(}1+\frac{\sigma}{a-\tilde{\v}}\Big{)}=\frac{\tilde{\v}}{a-\tilde{\v}}f(R)=v.
\end{eqnarray*}
Also, we can choose $c\mathscr{W}$, $c\neq 0$, instead of $\mathscr{W}$. Similar calculation shows that this also does not affect the final formulas, and so $u,v$ are defined uniquely. Sure, we must be careful as far as ramification is concerned - for abelian flows, for example, $\alpha$ is an abelian integral, and generally it is multi-valued. To avoid these problems, we must consider a pair $(u^{z},v^{z})$ for $z$ small enough. This finishes the proof.\\

There exists the second, a direct but more tedious way to prove (\ref{z}). Indeed, the boundary conditions show that at $z=0$ both sides of (\ref{z}) coincide. Now, let us take the derivative with respect to $z$ of the left side. This gives
\begin{eqnarray*}
-\frac{(v^{z})'}{(v^{z})^2}f\Big{(}\frac{u^z}{v^z}\Big{)}+f'\Big{(}\frac{u^z}{v^z}\Big{)}\frac{(u^{z})'v^{z}-u^{z}(v^{z})'}{(v^{z})^3}.
\end{eqnarray*}
When $z=0$, this becomes
 \begin{eqnarray*}
-\frac{\varrho}{y^2}f\Big{(}\frac{x}{y}\Big{)}+f'\Big{(}\frac{x}{y}\Big{)}\frac{y\varpi-x\varrho}{y^3}.
\end{eqnarray*}
According to (\ref{f-eq}), this is equal to $-1$, exactly the value of the derivative at $z=0$ of the right hand side of (\ref{z}). We then act by an induction and verify that (\ref{PDE-u}) imply that higher derivatives at $z=0$ of the left hand side of (\ref{z}) vanish.
\end{proof}

\subsection{Algebraicity}
\label{algebraicity}
Now, suppose $u\bl v$ is an algebraic projective plane flow. Then from (\ref{z}) we infer that
\begin{eqnarray}
f\Big{(}\frac{u(xz,yz)}{v(xz,yz)}\Big{)}=\frac{v(xz,yz)}{yz}f\Big{(}\frac{x}{y}\Big{)}-v(xz,yz).\label{mono}
\end{eqnarray}
This is the same as the first displayed Eq. on p. 756 in \cite{alkauskas-ab}. We now repeat the same argument as the one given in \cite{alkauskas-ab}, p. 755. Indeed, according to the fact (\ref{not-constant}), there exists a pair $(x,y)$ such that the function $\frac{u(xz,yz)}{v(xz,yz)}$ is a non-constant algebraic function in $z$. The right hand side of (\ref{mono}), call it $T(z)$, is an algebraic function in $z$, so is the argument on the left hand side; call it $L(z)$. We know that $L(z)$ is not a constant. We can write this as $f(L(z))=T(z)$, or $f(z)=T(L^{-1}(z))$. The inverse of an algebraic function is an algebraic itself: if $G(\omega,z)\in\mathbb{R}[\omega,z]$ and $G(L(z),z)=0$, then $G(z,L^{-1}(z))=0$. This shows that $f$ is also algebraic! This is a crucial consequence of the Eq. (\ref{f-eq}) for algebraic (and rational) flows.

\subsection{Basic lemma}
As now is clear, the following result is crucial in describing algebraic and rational projective flows.
\begin{lem}
\label{bas-lem}
Let $A(x)$ and $B(x)\neq 0$ be rational functions. Suppose, all the solutions to the ODE
\begin{eqnarray}
f(x)A(x)+f'(x)B(x)=1
\label{g-eq}
\end{eqnarray}
are algebraic functions. Then the general solution to this ODE can be given by
\begin{eqnarray*}
f(x)=r(x)+\sigma q^{1/N}(x),
\end{eqnarray*}
where $\sigma\in\mathbb{R}$, $N\in\mathbb{N}$, and $r(x),q(x)$ are rational functions. 
\end{lem}
\begin{proof}First, assume that $A\neq 0$. If every solution to (\ref{g-eq}) is rational, then the claim is obvious, and follows from general properties of non-homogeneous linear differential equations, $N=1$ in this case. Suppose, $f(x)$ is an algebraic solution to (\ref{g-eq}) of degree $N\geq 2$. Let its minimal monic equation over $\mathbb{R}(x)$ be given by
\begin{eqnarray*}
f^{N}+f^{N-1}\ell_{1}+\cdots+f\ell_{N-1}+\ell_{N}=0,\quad\ell_{i}\in\mathbb{R}(x).
\end{eqnarray*} 
Consider the function $h=f+\frac{\ell_{1}}{N}$. It satisfies the monic minimal equation
\begin{eqnarray}
h^{N}+h^{N-s}k_{s}+h^{N-s-1}k_{s+1}+\cdots+h k_{N-1}+k_{N}=0,
\quad k_{i}\in\mathbb{R}(x),\quad s\geq 2.
\label{h-alg}
\end{eqnarray}
The function $h$ satisfies the ODE (\ref{g-eq}) too, but with the right side now equal to some rational function $t(x)$. There are two cases to consider - \textbf{(i)} this rational function is equal to $0$, and \textbf{(ii)} it is not.\\

\textbf{(i)} We need to show that if all solutions to 
\begin{eqnarray*}
h(x)A(x)+h'(x)B(x)=0
\end{eqnarray*}
are algebraic, then $h(x)=\sigma q^{1/N}(x)$, where $q$ is a rational function, $N\in\mathbb{Z}$. This part follows easily from calculus routines. Indeed, we have:
\begin{eqnarray*}
\frac{h'}{h}=-\frac{A}{B}\Longrightarrow h=\sigma\exp\Big{(}-\int\frac{A}{B}\d x\Big{)}.
\end{eqnarray*}
Expressing a rational function $-\frac{A}{B}$ as a sum over integral part and primitive fractions, we see that the right hand side is algebraic only if
\begin{eqnarray*}
-\frac{A}{B}=\frac{1}{N}\sum\limits_{j=1}^{s}\frac{m_{j}}{x-\xi_{j}},\quad m_{j}\in\mathbb{Z},\quad \xi_{j}\in\mathbb{C}.
\end{eqnarray*}
Thus,
\begin{eqnarray*}
h^{N}=\sigma^{N}\prod\limits_{j=1}^{s}(x-\xi_{j})^{m_{j}}=\sigma^{N} q(x)\in\mathbb{C}(x).
\end{eqnarray*}
In fact, it belongs to $\mathbb{R}(x)$: if $\xi_{j}$ and $\xi_{k}$ are non-real complex conjugates, the integers $m_{j}$ and $m_{k}$ coincide, since $A$, $B$ are defined over $\mathbb{R}$. Thus, $f(x)=-\frac{\ell_{1}(x)}{N}+\sigma q^{1/N}(x)$, and the claim of the lemma is proved.\\

\textbf{(ii)} Now, suppose $t(x)\neq 0$. Dividing by $t(x)$, we obtain that $h(x)$ satisfies the ODE (\ref{g-eq}) with $\frac{A}{t}$ and $\frac{B}{t}$ instead of $A$ and $B$, which we, without loss of generality, rename again as $A$ and $B$. We will show that this leads to a contradiction. Indeed, let us differentiate (\ref{h-alg}) with respect to $x$. This gives
\begin{eqnarray}
\Big{(}Nh^{N-1}+(N-s)h^{N-s-1}k_{s}+\cdots+k_{N-1}\Big{)}h'+
h^{N-s}k'_{s}+h^{N-s-1}k'_{s+1}+\cdots+k'_{N}=0.
\label{beveik}
\end{eqnarray}
Now, due to (\ref{g-eq}), we know that $h'=\frac{1-hA}{B}$. Let us plug this into (\ref{beveik}). This gives
\begin{eqnarray*}
\Big{(}Nh^{N-1}+(N-s)h^{N-s-1}k_{s}+\cdots+k_{N-1}\Big{)}\frac{1-hA}{B}+
h^{N-s}k'_{s}+h^{N-s-1}k'_{s+1}+\cdots+k'_{N}=0.
\end{eqnarray*}
This is the equation for $h$ of degree $N$. Therefore, it is a rational multiple of the minimal equation (\ref{h-alg}). However, the coefficient at $h^{N}$ is equal to $-\frac{NA}{B}$, and since $s\geq 2$, the coefficient at $h^{N-1}$ is equal to $\frac{N}{B}$. This is not $0$, and we get a contradiction. \\

If $A=0$, the claim follows directly from the fact that if $\int\frac{\d x}{B}$ is algebraic, then it is rational.
 \end{proof}
\section{Corrections and alternative arguments}
In this section we list what changes in the main text of papers in question, and in next sections we prove some supplementary results. 
\subsection{The projective translation equation and rational plane flows. I }
\label{cor-rat}
Suppose, $\phi=u\bl v$ is a rational projective flow. Consider the ODE (\ref{f-eq}) (to be consequent, now we take $-1$ instead of $+1$). As we now know from Subsection \ref{algebraicity} and Lemma \ref{bas-lem}, it always has at least one rational solution. With its help we can find a $1$-BIR $\ell$ such that $\ell^{-1}\circ\phi\circ\ell(\m{x})$ is a flow in an univariate form $\mathcal{U}(x)\bl\frac{y}{y+1}$, using the terminology introduced in \cite{alkauskas}. Note that since  $\frac{y}{y+1}$ is a birational transformation of $P^{1}(\mathbb{R})$, $\mathcal{U}$ is necessary a Jonqui\`{e}res transformation of the form 
\begin{eqnarray*}
\mathcal{U}(x,y)=\frac{a(y)x+b(y)}{c(y)x+d(y)},
\end{eqnarray*}
$a,b,c,d\in\mathbb{R}[x]$, $ad-bc$ is not identically $0$.
The first coordinate of a vector field for $\phi$, given by $\frac{\d}{\d z}\frac{\mathcal{U}(xz,yz)}{z}|_{z=0}$, then must be a quadratic form; see \cite{alkauskas}, p. 307 for a calculation. Thus, in (\cite{alkauskas}, Section 4) we can immediately pass to Step IV. As mentioned Step I. is wrong. Indeed, Proposition 9, as stated, is not correct. Steps II. (apart from calculation of a vector field of $\mathcal{U}$) and III. turn out to be superfluous; in the next Subsection we will see that step IV. is also superfluous! However, as mentioned, there is a bright side of this: while working on Step III. we discovered 3 elliptic unramified flows, one of which turned out to be the simplest irreducible superflow whose group of symmetry is the dihedral group $\mathbb{D}_{3}$ of order $6$. The last subject has already ramified into a separate area of research; see \cite{alkauskas-super1,alkauskas-super2,alkauskas-super3}.  \\

Thus, all the results in \cite{alkauskas}, apart from Subsection 4.1, remain valid. To be more clear, the results in Subsection 4.1 are also correct, apart from this logical flaw and its consequences, and what is proven there can be formalised, but it has no immediate relevance to the problem of projective flows under consideration. 

\subsection{Algebraic and abelian solutions to the projective translation equation}
\label{cor-alg}
 This paper is devoted to special examples. So it remains valid, with few small corrections in formulation of classifying Theorems and one sentence in Section 3.\\

 Theorem 3 remains valid, if we change the first sentence as follows: ``Let $\phi(\m{x})$ be an algebraic flow, $\phi(\m{x})\neq \phi_{\mathrm{id}}(\m{x})$, whose vector field with the help of $1$-BIR can be transformed into a pair of quadratic forms". \\
 
 As we can now easily see, the variety of algebraic flows is much richer, where a vector field does not necessarily reduce to a pair of quadratic forms. There is a simple method to produce algebraic flows. Indeed, we can act as follows.\\
 
  Let $r(x)$, $q(x)$ be arbitrary rational functions over $\mathbb{R}$, and $N\in\mathbb{N}$. We can find an ODE (\ref{f-eq}) which is satisfied by the function $r(x)+\sigma q^{1/N}(x)$. For this we need to solve the system
\begin{eqnarray*}
\left\{\begin{array}{c}
r(x)\varrho(x)+r'(x)(x\varrho(x)-\varpi(x))=1,\\
Nq(x)\varrho(x)+q'(x)(x\varrho(x)-\varpi(x))=0,
\end{array}
\right.
\end{eqnarray*}
where now $\varpi$ and $\varrho$ are unknown rational functions. This has a unique rational solution $(\varpi,\varrho)$ if $rq'\neq N r'q$. In other words, if $r\neq \delta q^{1/N}$. If this is the case, we get an algebraic flow from the system (\ref{impl}), where $f(x)=r(x)$, and $\mathscr{W}(x,y)=y^{N}q^{-1}(\frac{x}{y})$. 
This algebraic flow has the vector field exactly equal to $y^{2}\varpi(\frac{x}{y})\bl y^{2}\varrho(\frac{x}{y})$.\\

Further, as noted in the abstract, we can now classify algebraic flows with respect to a $1$-BIR conjugacy, and this turns out to be even simpler than the classification of vector fields consisting of pairs of quadratic forms that produce algebraic flows (Theorem 3 in \cite{alkauskas-ab}).\\
    
  As a consequence of basic Lemma \ref{bas-lem}, we have the following. 
 \begin{prop}Suppose $x\varrho-y\varpi\neq 0$. The projective flow with the vector field $\varpi\bl\varrho$ is algebraic if and only if all the solutions of the ODE (\ref{f-eq}) are algebraic. If this is the case, then at least one of the solutions is rational.\\
 
  Let $\phi=u\bl v\neq\phi_{\mathrm{id}}$ be a projective algebraic plane flow, as before $x\varrho-y\varpi\neq 0$. Then there exists a $1$-BIR $\ell$, such that $\ell^{-1}\circ\phi\circ\ell$ is given by
\begin{eqnarray*}
\mathcal{U}(x,y)\bl\frac{y}{y+1},
\end{eqnarray*} 
where $\mathcal{U}(x,y)$ is an algebraic function. Thus, the vector field of this flow is $\ell$-conjugate to $\varpi(x,y)\bl (-y^{2})$, where $\varpi$ is a $2$-homogeneous rational function. 
 \end{prop} 
 So, suppose that $\varrho(x,y)=-y^2$. Similarly as with rational flows, we may call an algebraic flow of the form $\mathcal{U}(x,y)\bl\frac{y}{y+1}$ as being in an \emph{univariate form}. In this case $\varrho(x)=-1$, and one of the solutions of (\ref{f-eq}) is given by $f(x)=-1$. The second equation of (\ref{impl}) now reads as $\frac{1}{v}=\frac{1}{y}+1$, which does hold, since $v(x,y)=\frac{y}{y+1}$. Therefore, we get the following very easy characterization of algebraic flows. 
\begin{prop}
\label{prop-char}
Let $\mathscr{W}(x,y)$ be a homogeneous rational function of degree $N$, which cannot be written as $\mathscr{V}^{m}(x,y)$ for homogeneous rational $\mathscr{V}$ and an integer $m\geq 2$. Then the solution of the algebraic equation
\begin{eqnarray}
\mathscr{W}\Big{(}\mathcal{U}(x,y),\frac{y}{y+1}\Big{)}=
\mathscr{W}(x,y),
\label{eq-w}
\end{eqnarray}
which satisfies the boundary condition $\lim\limits_{z\rightarrow 0}\frac{\mathcal{U}(xz,yz)}{z}=x$, defines an algebraic  projective flow $\mathcal{U}(x,y)\bl\frac{y}{y+1}$ of level $N$ with a vector field
\begin{eqnarray*}
\frac{Ny\mathscr{W}}{\mathscr{W}_{x}}-xy\bl(-y)^2.
\end{eqnarray*}
\end{prop}
It seems a fascinating result that the above algebraic equation yields a solution to (\ref{funk}). This result eluded us while working on \cite{alkauskas, alkauskas-un, alkauskas-ab}. It is essentially used in \cite{alkauskas-comm} where commutative projective plane flows are given a very transparent classification. Proposition \ref{prop-char} now can be double-verified directly as follows. Indeed, let $\phi=\mathcal{U}(x,y)\bl\frac{y}{y+1}$, and $\phi^{z}=z^{-1}\phi(\m{x}z)=X\bl Y$. Then
\begin{eqnarray*}
X=\frac{\mathcal{U}(xz,yz)}{z},\quad Y=\frac{y}{yz+1}.
\end{eqnarray*}
Because $\mathscr{W}$ is $N$-homogeneous and because of (\ref{eq-w}), we have:
\begin{eqnarray*}
\mathscr{W}\Big{(}\frac{\mathcal{U}(Xw,Yw)}{w},\frac{Y}{Yw+1}\Big{)}&=&
\frac{1}{w^{N}}\mathscr{W}\Big{(}\mathcal{U}(Xw,Yw),\frac{Yw}{Yw+1}\Big{)}=
\frac{1}{w^{N}}\mathscr{W}(Xw,Yw)\\
&=&\mathscr{W}(X,Y)
=\frac{1}{z^{N}}\mathscr{W}\Big{(}\mathcal{U}(xz,yz),\frac{yz}{yz+1}\Big{)}\\
&=&\frac{1}{z^{N}}\mathscr{W}(xz,yz)=\mathscr{W}(x,y)\\
&=&\mathscr{W}\Bigg{(}\frac{\mathcal{U}\big{(}x(z+w),y(z+w)\big{)}}{z+w},\frac{y}{y(z+w)+1}\Bigg{)}.
\end{eqnarray*} 
Now, it is enough to note that
\begin{eqnarray*}
\frac{Y}{Yw+1}=\frac{y}{y(z+w)+1}.
\end{eqnarray*}
Thus, for $z,w$ small enough, with a correct choice of a branch, we have
\begin{eqnarray*}
\frac{\mathcal{U}(Xw,Yw)}{w}=\frac{\mathcal{U}\big{(}(z+w)x,(z+w)y\big{)}}{z+w}.
\end{eqnarray*}
And this is just the first coordinate of the identity (\ref{funk}), if written as (for ramified flows this is preferable)
\begin{eqnarray*}
\phi^{w}\circ\phi^{z}=\phi^{z+w},\quad \phi^{z}(\m{x})=z^{-1}\phi(z\m{x}).
\end{eqnarray*}
This works for any $N$-homogeneous function $\mathscr{W}$ - if $\mathcal{U}$ is defined from (\ref{eq-w}), $\mathcal{U}(x,y)\bl\frac{y}{y+1}$ is a projective flow. However, only for algebraic flows we are in the domain of algebraic geometry we are primarily concerned with - birational transformations and rational vector fields.

\begin{Example}
Let $\mathscr{W}(x,y)=x+\sigma y$. This gives
\begin{eqnarray*}
\mathcal{U}(x,y)=\frac{\sigma y^2+xy+x}{y+1}.
\end{eqnarray*}
This function is exactly denoted by  $\mathcal{W}^{(1)}_{\sigma,0}(x,y)$ in (\cite{alkauskas}, top of p. 287).
\end{Example}
Now, suppose that we have a rational flow $\mathcal{U}(x,y)\bl\frac{y}{y+1}$ in an univariate form. This means that the equation (\ref{eq-w}) is a first degree equation in $\mathcal{U}$. One of the possibilites is 
\begin{eqnarray*}
\mathscr{W}(x,y)=\frac{xy^{N}+ay^{N+1}}{x+by},\quad a,b\in\mathbb{R},\quad a\neq b.
\end{eqnarray*}  
Solving (\ref{eq-w}), we obtain
\begin{eqnarray*}
\mathcal{U}(x,y)=\frac{b(x+ay)(y+1)^{N}-a(x+by)}{-(x+ay)(y+1)^{N}+(x+by)}\cdot\frac{y}{y+1}.
\end{eqnarray*}
For $a=\frac{\sigma}{N-\sigma\tau}$, $b=-\frac{1}{\tau}$, we get exactly the flow, which is denoted by $\mathcal{W}^{(N)}_{\sigma,\tau}$ in (\cite{alkauskas}, p. 289, Eq. (28)). Thus, not only steps II. and III. are superfluous in proving Theorem \ref{mthm}, but also Step IV. (see \cite{alkauskas}, Sect. 4.4), though it sheds some light into the problem, is also unnecessary! Calculations in (\cite{alkauskas}, p. 289-290) show that any rational flow in an univariate form is $1$-BIR conjugate to canonical solution $\phi_{N}$ of the same level $N$. 
 \begin{Example} Consider the vector field
 \begin{eqnarray*}
-\frac{x(3x^5+x^3y^2+2y^5)}{3(x^2+y^2)^2}\bl-\frac{y(3y^5+y^3x^2+2x^5)}{3(x^2+y^2)^2}.
\end{eqnarray*}
The basic ODE (\ref{f-eq}) has the solution
\begin{eqnarray*}
-\frac{x^2+1}{x^3}+\sigma\frac{x^5+x^3-x^2-1}{x^3}.
\end{eqnarray*}
Thus, we can find the flow itself from the equations, as given by (\ref{impl}). That, is,
\begin{eqnarray*}
\frac{u^3v^3}{u^5+u^3v^2-u^2v^3-v^5}=\frac{x^3y^3}{x^5+x^3y^2-x^2y^3-y^5},\quad 
\frac{u^2+v^2}{u^3}=\frac{x^2+y^2}{x^3}+1.
\end{eqnarray*}
Solving with MAPLE, we obtain the flow $u(x,y)\bl v(x,y)=u(x,y)\bl u(y,x)$, where
\begin{eqnarray*}
u(x,y)=\frac{x^3}{x^3+x^2+y^2}+\frac{xy^2}{(x^3+x^2+y^2)^{1/3}(y^3+x^2+y^2)^{2/3}}.
\end{eqnarray*}
Thus we described exactly the algebraic flow given as Example (9) in \cite{alkauskas-0}, p. 338  (this is repeated as Example 3 in \cite{alkauskas-ab}), but from the other end, starting from the vector field and using Theorem \ref{thm1}.\\

Now we will transform the flow $u(x,y)\bl u(y,x)$ into an univariate form. Let $A(x,y)=\frac{x^2y+y^3}{x^3}$; so, again, we choose now the right hand side of (\ref{f-eq}) to be $-1$. Calculating the vector field $\varpi'\bl\varrho'$ according to Proposition \ref{prop1}, gives
\begin{eqnarray*}
\varpi'=-\frac{4}{3}xy+\frac{1}{3}\frac{y^4}{x^2},\quad \varrho'=-y^2.
\end{eqnarray*} 
Now, the solution to (\ref{w-eq}) in this case is equal to $\frac{y^4}{x^3-y^3}$. This gives
\begin{eqnarray*}
\frac{\frac{y^4}{(y+1)^{4}}}{\mathcal{U}^{3}-\frac{y^3}{(y+1)^3}}=\frac{y^4}{x^3-y^3}\Longrightarrow \mathcal{U}(x,y)=\frac{(x^3+y^4)^{1/3}}{(y+1)^{4/3}}.
\end{eqnarray*}
 \end{Example}
 In \cite{alkauskas-comm} we explore projective flows with rational vector fields such that they commute. It appears that this can happen, apart from level $0$ rational flows which always commute, only for pairs of special algebraic flows of level $1$.\\
 
 Case I of Theorem 2 in \cite{alkauskas-ab} describes all abelian flows whose vector field is a pair of quadratic forms. Of course, now we see that the variety of abelian flows is far richer. And so Diagram in Figure 1 can be interpreted just as arithmetic classification of vector fields given by a pair of quadratic forms.\\
 
In \cite{alkauskas-ab}, we asked for classification of algebraic and abelian flows with rational vector fields in any dimension; see Problem 9 and Problem 10 (\cite{alkauskas-ab}, p. 737). The question when orbits are algebraic curves in higher dimensions may be very hard. For example, Jouanolou \cite{jouanolou} showed that the vector field $y^2\bl z^2\bl x^2$ does not have a rational first integral, so \emph{a posteriori} it does not have two integrals, and the flow is not abelian. Of course, classifying $3$-dimensional $2$-homogeneous rational vector fields with $2$ independent algebraic first integrals might be easier, though, as a contrast, the question on finding all algebraic flows might be even simpler with the results of the current paper in disposition, and this is demonstrated in Section \ref{induct}.\\  
 
There is a small mistake in the formulation of (\cite{alkauskas-ab}, Proposition 10). Of course in order the flow with a vector field given by a pair of integral quadratic forms $\varpi\bl\varrho$ to be non-abelian, it is not enough that $y\varpi-x\varrho$ has multiple roots, as the example $\varpi\bl\varrho=2xy\bl y^2$, $\phi=\frac{x}{(1-y)^2}\bl\frac{y}{1-y}$ shows. It is important that $\frac{\varrho}{y\varpi-x\varrho}$, after a possible reduction, still has a denominator with multiple roots. The correct formulation should read as follows.
\begin{itemize}
\item[$\star$]{\it Let $\varpi(x,y)\bl\varrho(x,y)$ be a pair of quadratic forms with integral (or rational, the conclusion is the same) coefficients that gives rise to a non-abelian flow. Then there exists a linear change such that...}
\end{itemize}  
 and the rest of the claim is as presented, with $\lambda\in\mathbb{Q}$. \\
 
 Indeed, if $y\varpi-x\varrho$ is a cube of a linear polynomial, we can linearly conjugate to achieve this cube to be equal to (a scalar multiple of) $y^3$. Thus, we may suppose
\begin{eqnarray*}
\varpi\bl\varrho=ax^2+bxy+cy^2\bl axy+by^2,\quad c\neq 0.
\end{eqnarray*}
Note that $a=0$ cannot happen, since then $\frac{\varrho}{y\varpi-x\varrho}$ has no mulitple roots in the denominator. The case $b\neq 0$ gives, after a linear conjugation, the item i) of Proposition 10, and the case $b=0$ - the item ii). Now, suppose $y\varpi-x\varrho$ factors into linear factors as $\ell^{2}k$, $\frac{\ell}{k}$ not a constant. Then after a linear conjugation we can achieve that $\varpi\bl\varrho=ax^2+bxy\bl cxy+dy^2$ (\cite{alkauskas}, Subsection 4.3). Then $y\varpi-x\varrho=xy((a-c)x+(b-d)y)$. This double root cannot be $y^2$, the root of $\varrho$, since then, once again, $\frac{\varrho}{y\varpi-x\varrho}$ does not contain a double root in the denominator. So this double root is $x$, and this implies $b=d$. This, after a linear conjugation, gives an item iii) in Proposition 10.\\

 To finish comments on \cite{alkauskas-ab}, note that neither in \cite{alkauskas-ab}, nor here we give a complete characterization of algebraic flows up to $1$-BIR equivalence, contrary to rational case described by Theorem \ref{mthm}. This problem, among other aspects of algebraic and abelian flows - solenoidal, symmetric algebraic flows, flows in an univariate form in a given algebraic function field, algebraic flows which share orbits (see the beginning of Subsection \ref{orbits-share}), addition formulas for abelian functions on orbits, elliptic abelian flows - is treated in \cite{alkauskas-ab2}. 

\subsection{The projective translation equation and unramified $2$-dimensional flows with rational vector fields}
Theorem 1 in \cite{alkauskas-un} remains valid, if we replace the phrase \emph{``is a pair of $2$-homogenic rational functions"} with \emph{``is a pair of quadratic forms"}.\\

Moreover, we can add more. Let $\phi$ be a projective flow with a rational vector field. There exists a $1$-BIR $\ell$ such that $\ell^{-1}\circ\phi\circ\ell$ has a vector field whose second coordinate vanish if and only if $\phi$ is an abelian flow of level $1$. This fact was of utmost importance in finding all commuting projective flows in \cite{alkauskas-comm}. Indeed, as is clear from Proposition \ref{prop1}, $\varrho'=0$ (prime, not derivative) for a certain rational $0$-homogeneous $A$ if and only if $1$-homogeneous solution to (\ref{w-eq}) is rational, and this happens exactly for abelian flows of level $1$. This case is treated in \cite{alkauskas-un}, Subsection 4.4. So, the phrase which is being inserted can read as ``is a pair of quadratic forms, or $\phi$ is an abelian flow of level $1$". \\ 

 At the moment it is not clear how one should approach the question of ramification of projective flows with general $2$-homogeneous vector fields, except in special examples. The very notion of ``unramified flow" seems to be new and important \cite{conlon}; for example, the notion of time parameter to vary over $\mathbb{C}$ as opposed to $\mathbb{R}$ is also non-standard practice in differential geometry, and revealing new phenomena, like ramification. For example, the vector field $x^2-xy\bl y^2-2xy$, briefly treated in \cite{alkauskas-un}, Subsecion 4.3, was shown to be indeed unramified in the Appendix to \cite{alkauskas-ab}. The proof of this involves rational addition formulas for elliptic functions with hexagonal period lattice, also birational transformations, and basic algebra of the number field $\mathbb{Q}(\sqrt{3})$. Thus, the proof is not trivial at all! The discovery of unramified non-rational flows is the second most important positive side-result of the flaw. This question is treated in \cite{alkauskas-un2}. For example, if $n$-dimensional vector field is given by a collection of  quadratic forms, we can act in a similar way as we did in \cite{alkauskas}, Subsection 4.3. The theory of unramified flows is in a way dual to the theory of algebraic flows (see Section \ref{induct}), since the intersection of the set of unramified flows and the set of algebraic flows is exactly the set of rational flows.\\

We also add what is now known concerning various problems described in \cite{alkauskas-un}, Section 5.\\

\underline{\emph{Problem 2}} was solved in case $N=3$ in \cite{alkauskas-super1}. The explicit formulas involve Jacobi elliptic functions. Also, there exists an irreducible representation of $\Sigma_{N+1}\oplus\mathbb{Z}_{2}$ of dimension $N+1$, and this leads to the superflow, whose symmetry group is of order $2(N+1)!$. In \cite{alkauskas-super2} we pose Problem 13 asking whether this is in fact the smallest possible symmetry group of a superflow in dimension $(N+1)$, $N\geq 2$.\\

\underline{\emph{Problem 3}} was solved, as mentioned, for groups $O(2)$, $O(3)$, and is under development for $U(2)$; see \cite{alkauskas-super1,alkauskas-super2,alkauskas-super3} (the corresponding dimensions of Lie algebras are, respectively, $1$, $3$ and $4$). There is a tiny slip in the terminology in \cite{alkauskas-un}, p. 906. $\Sigma_{4}$ and $\Sigma_{4}'$ are not contragradient representations, but one is obtained from the other tensoring with a $1$-dimensional non-trivial (determinant) representation.\\

 Next, I would like the reader to pay attention to p. 906. One of the excerpts reads as follows.
\begin{itemize}
\item[$\bigstar$]{\it One can check directly that the vector field
\begin{eqnarray*}
\mathbf{Q}(\m{x})=
\frac{y^3z-yz^3}{x^2+y^2+z^2}\bl\frac{z^3x-zx^3}{x^2+y^2+z^2}\bl\frac{x^3y-xy^3}{x^2+y^2+z^2}
\end{eqnarray*}
is invariant under conjugation with all matrices from $\Sigma'_{4}$. Up to the constant factor, this is the only vector field whose common denominator} (a typo in \cite{alkauskas-un}, not numerator) {\it is of degree at most $2$; so, integrating this vector field (solving the same PDE with different boundary conditions), we get a $\Sigma'_{4}$-superflow, essentially different from the $\Sigma_{4}$-superflow. One might wonder what new special functions occur in the analytic formulas. We leave this very promising side of investigations of the projective translation equation for the future.}
\end{itemize}
This problem was solved in \cite{alkauskas-super1}. It appears that, generically, the orbits of this superflow are space curves of genus $9$, given by $\{x^2+y^2+z^2=a, x^4+y^4+z^4=b, a,b\in\mathbb{R}\}$. Via a reduction of the differential system, these curves and abelian functions on them can be reduced to genus $2$, and yet with another reduction (this works like a miracle) the genus can be lowered to $1$, and thus this flow can be explicitly integarted in terms of Weierstrass elliptic functions, but in a rather complicated way. For example, let $\xi=\frac{b}{a^2}$. Then this elliptic function has a square period lattice only in two cases
\begin{eqnarray*}
\xi=\frac{5}{9},\text{ or }\xi=\frac{1}{18}\sqrt[3]{16\sqrt{2}+13}-\frac{1}{18}\sqrt[3]{16\sqrt{2}-13}+\frac{7}{18}=0.4535087845_{+}.
\end{eqnarray*}     

\emph{Problem 5} is treated in \cite{alkauskas-dim3} (see the current paper, Section \ref{induct}). \emph{Problem 6} is treated in \cite{alkauskas-un2}, and \emph{Problem 7} - in \cite{alkauskas-quasi}. \emph{Problem 4} and \emph{Problem 8} are left for the future.   
\section{Inductive construction of algebraic or rational projective flows}
\label{induct}
Next, we continue with additions to \cite{alkauskas, alkauskas-ab}. In this section we show that the results of this paper put investigations of higher dimensional rational and algebraic projective flows in a much more optimistic perspective, as compared to the one described in (\cite{alkauskas}, Subsection 5.6).\\

As a generalization to calculations just after Proposition \ref{prop-char}, we prove the following second most important (after Theorem \ref{thm1}) result of this paper. Now we choose three ``space" variables to be $x,y,z$, and ``time" parameters to be $s,t$. 
\begin{thm}
\label{thm2}
Suppose, $\phi(\m{x})=u(x,y)\bl v(x,y)$ is a $2$-dimensional projective flow. Let $N\in\mathbb{N}$, and suppose $\mathscr{W}$ is a $N$-homogeneous function in $3$ variables. Let us define the function $\mathcal{T}(x,y,z)$ by the equation
\begin{eqnarray}
\mathscr{W}\big{(}u(x,y),v(x,y),\mathcal{T}(x,y,z)\big{)}=\mathscr{W}(x,y,z)
\label{constr}
\end{eqnarray}
and the boundary condition $\lim\limits_{t\rightarrow 0}\frac{\mathcal{T}(xt,yt,zt)}{t}=z$. Then
\begin{eqnarray*}
\Phi(x,y,z)=u(x,y)\bl v(x,y)\bl \mathcal{T}(x,y,z)
\end{eqnarray*}
is a projective flow. If $\phi$ is an algebraic flow with rational vector field and $\mathscr{W}$ is a rational function, then $\Phi$ is algebraic flow (with rational vector field). If $\phi$ is a rational flow, $\mathscr{W}$ is rational, and, moreover, $\mathscr{W}$ is a linear-fractional expression for $z$, then $\Phi$ is a rational flow.
\end{thm}
\begin{proof}The proof mimics the calculations after Proposition \ref{prop-char}. Indeed, let
\begin{eqnarray*}
X=\frac{u(xt,yt)}{t},\quad Y=\frac{v(xt,yt)}{t},\quad 
Z=\frac{\mathcal{T}(xt,yt,zt)}{t}.
\end{eqnarray*} 
Because $\mathscr{W}$ is $N$-homogeneous and because of (\ref{constr}), we have:
\begin{eqnarray*}
& &\mathscr{W}\Big{(}\frac{u(Xs,Ys)}{s},\frac{v(Xs,Ys)}{s},\frac{\mathcal{T}(Xs,Ys,Zs)}{s}\Big{)}\\
&=&
\frac{1}{s^{N}}\mathscr{W}\Big{(}u(Xs,Ys),v(Xs,Ys),\mathcal{T}(Xs,Ys,Zs)\Big{)}=
\frac{1}{s^{N}}\mathscr{W}(Xs,Ys,Zs)\\
&=&\mathscr{W}(X,Y,Z)
=\frac{1}{t^{N}}\mathscr{W}\Big{(}u(xt,yt),v(xt,yt),\mathcal{T}(xt,yt,zt)\Big{)}\\
&=&\frac{1}{t^{N}}\mathscr{W}(xt,yt,zt)=\mathscr{W}(x,y,z)\\
&=&\mathscr{W}\Bigg{(}\frac{u\big{(}x(s+t),y(s+t)\big{)}}{s+t},\frac{v\big{(}x(s+t),y(s+t)\big{)}}{s+t},\frac{\mathcal{T}\big{(}x(s+t),y(s+t),z(s+t)\big{)}}{s+t}\Bigg{)}.
\end{eqnarray*} 
Now, in the above, the first term and the last term have the first and the second arguments that coincide, since $\phi$ is a projective flow. For $s,t$ small enough, due to the boundary condition, with the correct choice of the branch the third argument then will coincide, too. This shows that $u\bl v\bl\mathcal{T}$ is a projective flow. The rest of the statements of the Theorem are immediate, though very important! For example, suppose $\phi$ has a rational vector field $\varpi(x,y)\bl\varrho(x,y)$. Lets us differentiate
\begin{eqnarray*}
\mathscr{W}(u^{t}(x,y),v^{t}(x,y),\mathcal{T}^{t}(x,y,z))=\mathscr{W}(x,y,z) 
\end{eqnarray*} 
with respect to $t$, and put $t=0$. Let
\begin{eqnarray*}
\frac{\d}{\d t}\frac{\mathcal{T}(xt,yt,zt)}{t}\Big{|}_{t=0}=\sigma(x,y,z).
\end{eqnarray*}
This gives
\begin{eqnarray*}
\mathscr{W}_{x}\varpi+\mathscr{W}_{y}\varrho+\mathscr{W}_{z}\sigma=0.
\end{eqnarray*}
Of course, we know this in advance, since $\mathscr{W}=\mathrm{const.}$ is the first integral. This shows that $\sigma$ is rational and $2$-homogeneous, if $\mathscr{W}$ is rational.
\end{proof}
Now, let $\ell$ be a $1$-homogeneous birational transformation of $\mathbb{C}^{3}$. Then, as usual, $\ell^{-1}\circ(u\bl v\bl\mathcal{T})\circ\ell(x,y,z)$ is an algebraic projective flow with rational vector field, if $u\bl v\bl\mathcal{T}$ is such. Now, when we have constructed an algebraic flow $u(x,y,z)\bl v(x,y,z)\bl w(x,y,z)$, we can extend it to dimension $4$ with a help of homogeneous function $\mathscr{V}$ in $4$ variables, and so on. \\

To continue presenting open problems about projective flows, started in \cite{alkauskas-un}, and continued in \cite{alkauskas-ab, alkauskas-comm,alkauskas-super1,alkauskas-super2,alkauskas-super3}, we are lead to the following. But first, let $n\in\mathbb{N}$, $\m{x}=(x_{1},\ldots,x_{n})$, and let $J(\m{x})$ be any rational $1$-homogeneous function. A vector field $x_{1}J(\m{x})\bl x_{2}J(\m{x})\bl\cdots\bl x_{n}J(\m{x})$ produces a flow
\begin{eqnarray*}
\frac{x_{1}}{1-J(\m{x})}\bl\cdots\bl\frac{x_{n}}{1-J(\m{x})}.
\end{eqnarray*}
This is a direct analogue of rational flow of level $0$ in any dimension. If an algebraic flow of dimension $n$ has orbits as curves $\mathscr{W}_{1}=\mathrm{const.},\ldots,\mathscr{W}_{n-1}=\mathrm{const}.$, where $\mathscr{W}_{i}$ are homogeneous functions,  we can assume that either all the homogeneity degrees are $0$ (then we have the above case), or all of them are non-zero. Indeed, otherwise, if the homogeneity degree of $\mathscr{W}_{2}$ is $0$ and that of $\mathscr{W}_{1}$ is not, we replace $\mathscr{W}_{2}$ with $\mathscr{W}_{1}\cdot\mathscr{W}_{2}$. However, construction of Theorem \ref{thm2} works even in the case $N=0$. So, we ask the following.
\begin{prob14}Does the construction of Theorem \ref{thm2} and a subsequent conjugation with a $1$-BIR produce all algebraic projective flows with rational vector fields in dimension $3$, and, inductively, in any dimension?
\end{prob14}
This question is treated in detail \cite{alkauskas-dim3}. It is hopeful to solve it even without knowing the structure of the Cremona group of $P^{n-1}(\mathbb{K})$ (where $\mathbb{K}$ is $\mathbb{C}$ or $\mathbb{R}$) in higher dimensions. Indeed, $1$-homogeneous birational transformations of $\mathbb{K}^{n}$ are described in terms of Cremona group of $P^{n-1}(\mathbb{K})$. So, when $n=3$, this Cremona group (over $\mathbb{C}$) contains not only linear-fractional transformations, but also the so called \emph{quadratic transformations}, and the solution to the above question might depend which field we are working with.\\

Let now $n\in\mathbb{N}$. If the answer to the above question is affirmative, this will imply the following very strong corollary. Let $\Phi$ be an algebraic projective flow with a rational vector field. Then there exists a $1$-BIR $\ell$ such that $\Psi=\ell^{-1}\circ\Phi\circ\ell$ is an algebraic flow of the following form. Let $\m{x}_{i}=(x_{1},x_{2},\ldots,x_{i})$. Then
\begin{eqnarray*}
\Psi=u_{1}(\m{x}_{1})\bl u_{2}(\m{x}_{2})\bl\cdots\bl u_{n}(\m{x}_{n}).
\end{eqnarray*}
Since the flow $\Phi$ is algebraic, it has $(n-1)$-independent rational homogeneous first integrals $\mathscr{W}_{i}$, $1\leq i\leq n-1$. However, we have $(n-1)$ homogeneous equations for $n$ functions $u_{i}$, $1\leq i\leq n$; for example, such is the first equation of (\ref{impl}). So, one more non-homogeneous equation is needed, an analogue of the second equation of (\ref{impl}), to ensure that we have a projective flow. More importantly, this will characterise only abelian flows, so some deeper arithmetic is needed the claim that these equations produce not only abelian, but an algebraic flow.\\

However, the situation with a flow $\Psi$ is much simpler. Since $\m{x}_{1}=x_{1}$, then $u_{1}(x_{1})$ is a one dimensional projective flow with a rational vector field, and so
\begin{eqnarray*}
u_{1}(x_{1})=\frac{x_{1}}{1-cx_{1}}\text{ for a certain }c\in\mathbb{R}.
\end{eqnarray*}
Since $u_{2}(\m{x}_{2})=u_{2}(x_{1},x_{2})$, let $\mathscr{V}_{1}$ be the integral of the flow $u_{1}\bl u_{2}$. Then
\begin{eqnarray*}
\mathscr{V}_{1}\Big{(}\frac{x_{1}}{1-cx_{1}},u_{2}\Big{)}=\mathscr{V}_{1}(x_{1},x_{2}).
\end{eqnarray*}
Thus, this gives an algebraic equation for $u_{2}$, and the correct branch, compatible with a boundary condition, gives $u_{2}$. We proceed in the same way to get explicit equation for $u_{i}$, $2\leq i\leq n$. This demonstrates the strong consequences the solution to Problem 14 might have.   
\begin{Example} Consider the canonical flow of level $N$, only after conjugating with a linear involution $i_{0}(x,y)=(y,x)$. So, let $\widehat{\phi}_
{N}=u\bl v=\frac{x}{x+1}\bl y(x+1)^{N-1}$. Let $\mathscr{W}(x,y,z)=z(x^2+xy)$. Form the above construction (Theorem \ref{thm2}), we get a $3$-dimensional rational flow
\begin{eqnarray*}
\Phi_{N}=u(x,y)\bl v(x,y)\bl\mathcal{T}(x,y,z)=\frac{x}{x+1}\bl y(x+1)^{N-1}\bl
\frac{z(x+y)(x+1)^{2}}{x+(x+1)^{N}y}.
\end{eqnarray*}
Let $\ell(x,y,z)$ be a $1$-BIR
\begin{eqnarray*}
\ell(x,y,z)=x\bl y\bl\frac{yz}{x+y},\quad \ell^{-1}(x,y,z)=x\bl y\bl\frac{(x+y)z}{y}.
\end{eqnarray*}
By a direct calculation,
\begin{eqnarray*}
\ell^{-1}\circ\Phi_{N}\circ\ell(\m{x})=\frac{x}{x+1}\bl y(x+1)^{N-1}\circ z(x+1)^{2-N}.
\end{eqnarray*}
Now, $\mathrm{g.c.d.}(N,1-N)=1$. We have seen (\cite{alkauskas}, p. 325) that with the help of $1$-BIR involutions this flow can be transformed into the flow
\begin{eqnarray*}
\phi_{1}(x,y,z)=x\bl\frac{y}{z+1}\bl\frac{z}{z+1}.
\end{eqnarray*}  
So, as far as this example is concerned, the answer to question in Problem 5 in \cite{alkauskas-un} is affirmative.  
\end{Example}
We see that flows $\Phi_{N}$ are all $1$-BIR (in dimension $3$) conjugate, though flows $\phi_{N}$ (in dimension $2$) are not. Related to this, it is apt to state the result, first formulated in \cite{alkauskas-super2}, which shows that projective flows in dimensional $n$ might be considered as general flows in dimension $n-\frac{1}{2}$, and via (\ref{funk}) \emph{all} flows in $\mathbb{R}^{n-1}$, satisfying the translation equation $F(F(\m{x},s),t)=F(\m{x},s+t)$ might be described.
\begin{prop}
\label{prop-ner}
Any flow in $\mathbb{R}^{n}$ is a section on a hyperplane of a $(n+1)$-dimensional projective flow. In other words, for any flow $F(\m{x},t):\mathbb{R}^{n}\times\mathbb{R}\mapsto\mathbb{R}^{n}$, there exists a projective flow $\phi:\mathbb{R}^{n+1}\mapsto\mathbb{R}^{n+1}$, such that
\begin{eqnarray*}
F(\m{x},t)=\frac{1}{t}\phi\Big{(}(\m{x},1)t\Big{)}=\frac{1}{t}\phi\big{(}\m{x}t,t\big{)}.
\end{eqnarray*}
\end{prop}
This is just a consequence of the fact that any vector field in $\mathbb{R}^{n}$ is a section on the hyperplane $x_{n+1}=1$ of a $2$-homogeneous vector field in $\mathbb{R}^{n+1}$. Thus, (\ref{funk}) is not in this sense a special example of a translation equation, but rather a reduction of a dimension by $``\frac{1}{2}"$.  \\

We can say even more. To describe ``affine" flows in dimension $n$ we need $n$ space variables and $1$ time variable, $(n+1)$ variables in total. To describe an affine flow as a section of a projective flow on a hyperplane $x_{n+1}=1$, we also need $(n+1)$ variables. Thus, projective flows which are needed to describe affine flows are of the same level of complexity; in the special case when vector field is given by a collection of $2$-homogeneous rational functions, this complexity is even lower. The conclusion is that investigation of projective flows is a more efficient way to study affine flows! Moreover, in projective flow case time parameter is integrated into a collection of space parameters, and so the questions of algebraic, rational flows is unambiguous, whereas the same questions for affine flows are ambiguous. For example, what is the definition of ``affine rational flow" $F(x,t)$ in dimension $1$? Both variables are of a very different nature, so we may define rational flows as follows: for any fixed $t$, $F(x,t)$ is a rational function in $x$. According to this definition a $1$-dimensional vector field $\varpi=x$ produces a flow $F(x,t)=xe^{t}$, which, based on the above definition, is a rational flow. On the other hand, we may define affine flow with a rational vector field to be \emph{rational} or \emph{algebraic}, if it is a section of a projective flow which, respectively, is rational or algebraic. This definition is not the only possible, but it is unambiguous. 

\section{Symmetric rational flows}
The last three Sections are all in the spirit of (\cite{alkauskas}, Sect. 5.3) and describe three further various aspects of rational plane flows.\\

Here we repeat the definition from (\cite{alkauskas}, Sect. 5.3).
\begin{defin}
Let $i$ be a $1$-BIR involution. We say that a projective flows $\phi(\m{x})$ is \emph{$i$-symmetric}, if $\phi$ is invariant under conjugation with $i$: $i\circ\phi\circ i=\phi$.
\end{defin}

In \cite{alkauskas}, Appendix A.2 it is proved that there are four types of involutions. In Subsection 5.3 we found all rational flows symmetric with respect to a linear involution $i_{0}(x,y)=(y,x)$. Note that this subject has ramified significantly: in \cite{alkauskas-super1,alkauskas-super2,alkauskas-super3} we investigate flows symmetric with respect to bigger linear groups. If these flows satisfy the conditions of minimality and uniqueness, they are called \emph{superflows}. These all turn to be abelian flows. Some of them (reducible $2$-dimensional flows over $\mathbb{C}$) turn out to be algebraic, or even rational. For example, the  flow
\begin{eqnarray}
\phi_{\mathrm{sph},\infty}(x,y)=(x-y)^2+x\bl (x-y)^2+y
\label{sph}
\end{eqnarray}
(see \cite{alkauskas-0}, p. 338, and \cite{alkauskas}, p. 286) has an infinite group of symmetries, and is the superflow for a cyclic group of order $6$ generated by the matrix $\gamma=\begin{pmatrix}
\zeta & 0\\
\zeta+\zeta^{-1}&-\zeta^{-1}
\end{pmatrix}$, where $\zeta=e^{\frac{2\pi i}{3}}$ \cite{alkauskas-super3}. The fact that this superflow is reducible is non-trivial - it amounts to showing that its group of symmetries, whose Lie algebra is isomorphic to a unique $2$-dimensional non-commutative Lie algebra, has no finite non-commutative subgroups. Moreover, this rational flows $\phi_{\mathrm{sph},\infty}$ is also solenoidal; see Section \ref{solenoidal} of this paper.\\

In this Subsection we lay foundations for investigation of projective flows with rational vector fields whose group of symmetries are finite subgroups of a group of $1$-homogeneous birational transformations of $\mathbb{R}^{N}$, where $N$ is the dimension of a flow. We denote this group as in \cite{alkauskas}; that is, $Bir_{\mathrm{h}}(\mathbb{R}^{N})$.\\

So, the task of this Subsection is to find all rational flows symmetric with respect to a non-linear $1$-homogeneous involution $i(x\bl y)=\frac{y^2}{x}\bl y$, and henceforth $i$ has this fixed meaning. This involutions plays an important role in describing rational flows: for example, see \cite{alkauskas}, p. 289. \\

First, we will find the condition on the vector field so that it is invariant under conjugation with $i$. Note that $i=j\circ i_{0}$, where 
\begin{eqnarray*}
i_{0}(x\bl y)=y\bl x,\quad j(x\bl y)=\frac{x^2}{y}\bl x.
\end{eqnarray*} 
Now, $j$ is a $1$-BIR of the form (\ref{bir}). Suppose, a flow $\phi$ is invariant under conjugation with $i$. Let its vector field $\varpi\bl \varrho$. We get
\begin{eqnarray*}
i_{0}\circ j^{-1}\circ(\varpi\bl\varrho)\circ j\circ i_{0}=\varpi\bl \varrho.
\end{eqnarray*}
The vector field $j^{-1}\circ(\varpi\bl\varrho)\circ j$ can be calculated using The formula from Proposition \ref{conjug}. In our case, $A=\frac{x}{y}$. This gives
\begin{eqnarray*}
j^{-1}\circ(\varpi\bl\varrho)\circ j=\frac{x^2}{y^2}\varrho\bl\frac{2x}{y}\varrho-\varpi=\varpi'\bl\varrho'. 
\end{eqnarray*}
So, 
\begin{eqnarray*}
i_{0}\circ(\varpi'\bl\varrho')\circ i_{0}=\frac{2y}{x}\varrho(y,x)-\varpi(y,x)\bl\frac{y^2}{x^2}\varrho(y,x).
\end{eqnarray*}
Suppose $i\circ\phi\circ i=\phi$. Since the above is the vector field for the flow $i\circ\phi\circ i$, we get that
\begin{eqnarray*}
\frac{2y}{x}\varrho(y,x)-\varpi(y,x)\bl\frac{y^2}{x^2}\varrho(y,x)
=\varpi(x,y)\bl\varrho(x,y).
\end{eqnarray*}
Since $\varpi$ and $\varrho$ are $2$-homogeneous functions, these equalities can be rewritten as 
\begin{eqnarray*}
2xy\varrho\Big{(}\frac{y}{x},1\Big{)}=\varpi(x,y)+\varpi(y,x),\quad
\varrho\Big{(}\frac{y}{x},1\Big{)}=\varrho\Big{(}\frac{x}{y},1\Big{)}.
\end{eqnarray*}
Now we see that if the first equality is satisfied, the second is satisfied automatically. This situation is completely analogous (though more complicated) to the situation investigated in \cite{alkauskas}, Subsection 5.3. Namely, if the flow $\phi$ is invariant under conjugation with $i_{0}$, then this gives two equations $\varpi(x,y)=\varrho(y,x)$, $\varpi(y,x)=\varrho(x,y)$. So, in the latter case, any one of the two equations is the consequence of the other - this is rather a tautology.\\

 Therefore, in order the vector field $\varpi\bl\varrho$ to be invariant under conjugation with $i$, it is necessary and sufficient that  
\begin{eqnarray}
\frac{2x}{y}\varrho(x,y)=\varpi(x,y)+\varpi(y,x).
\label{symm}
\end{eqnarray}
\subsection{Rational flows of level $0$} Let $J(x,y)$ be any $1$-homogeneous rational function. We know that the vector field of the rational flow $\frac{x}{1-J(x,y)}\bl \frac{y}{1-J(x,y)}$ of level $0$ is given by $xJ(x,y)\bl y J(x,y)$. If this is $i$-invariant, the condition (\ref{symm}) gives $xJ(x,y)=yJ(y,x)$. Therefore, we obtain the following characterization. 
\begin{prop}All level $0$ rational $i$-symmetric flows can be given by the following construction. Let $r(t)$ be any rational function with the property $r(t)=r(\frac{1}{t})$. Then
\begin{eqnarray*}
\phi=\frac{x}{1-J(x,y)}\bl\frac{y}{1-J(x,y)},
\end{eqnarray*}
where $J(x,y)=yr(\frac{x}{y})$. 
\end{prop} 
We can verify directly that $i\circ\phi\circ i(\m{x})=\phi(\m{x})$.
\subsection{Rational flows of level $N\geq 1$} Now we can act exactly the same way as in \cite{alkauskas}, Subsection 5.3. However, this all turns out to be unnecessary, since $i$-symmetric rational flows can be obtain from $i_{0}$-symmetric flows as follows.\\

Suppose, rational flow $\phi$ is $i_{0}$-symmetric: $i_{0}\circ\phi\circ i_{0}=\phi$. Let us conjugate this identity with respect to some $1$-BIR $\ell$. We get
\begin{eqnarray}
\ell^{-1}\circ i_{0}\circ\phi\circ i_{0}\circ\ell=\ell^{-1}\circ\phi\circ\ell.
\label{focus}
\end{eqnarray}   
  
Now, suppose there exists a $1$-BIR $\ell$ such that
\begin{eqnarray}
i_{0}\circ\ell=\ell\circ i\Longrightarrow \ell^{-1}\circ i_{0}=i\circ\ell^{-1}.
\label{cond}
\end{eqnarray}
Thus, then (\ref{focus}) turns out to be 
\begin{eqnarray*}
i\circ\ell^{-1}\circ\phi\circ\ell\circ i=\ell^{-1}\circ\phi\circ\ell.
\end{eqnarray*} 
Thus, if $\phi$ is $i_{0}$-symmetric rational flow, $\ell^{-1}\circ\phi\circ\ell$ turns out to be $i$-symmetric. And vica versa - if $\phi$ is $i$-symmetric, then $\ell\circ\phi\circ\ell^{-1}$ is $i_{0}$-symmetric. Let $\ell(\m{x})=a(x,y)\bl b(x,y)$ be a $1$-BIR. The identity (\ref{cond}) is tantamount to 
\begin{eqnarray}
a\Big{(}\frac{y^2}{x},y\Big{)}=b(x,y).
\label{ab}
\end{eqnarray}
There are many solutions to this equation, but the crucial property we need is that $a\bl b$ is a $1$-BIR. We look for $a(x,y)\bl b(x,y)$ to be of the form (\ref{bir}), that is, $a\bl b=xA\bl yA$, where $A$ is $0$-homogeneous. Let $A(x,y)=r(\frac{x}{y})$. Then (\ref{ab}) is equivalent to $tr(t)=r(\frac{1}{t})$. We can choose $r(t)=\frac{1}{t+1}$. So, one of the $1$-BIR solutions to (\ref{cond}) is given by
\begin{eqnarray*}
\ell_{0}(\m{x})=\frac{xy}{x+y}\bl\frac{y^2}{x+y},\quad 
\ell_{0}^{-1}(\m{x})=\frac{x(x+y)}{y}\bl(x+y).
\end{eqnarray*}
Now, let $s$ by a $1$-BIR given by $xA\bl yA$, where $A(x,y)=A(y,x)$. Then we can verify that if $\phi$ is $i_{0}$-symmetric, then $s^{-1}\circ\ell_{0}^{-1}\circ\phi\circ\ell_{0}\circ s$ is $i$-symmetric. Indeed, any two $1$-BIR's of the form (\ref{bir}) commute (see \cite{alkauskas}, Proposition 15), so
\begin{eqnarray*}
\chi=s^{-1}\circ\ell_{0}^{-1}\circ\phi\circ\ell_{0}\circ s=
\ell_{0}^{-1}\circ s^{-1}\circ\phi\circ s\circ\ell_{0}.
\end{eqnarray*}
Since $\phi$ is $i_{0}$-symmetric, so is $s^{-1}\circ\phi\circ s$, and we have seen that the flow $\chi$ is then $i$-symmetric. Thus, we can directly formulate analogues of Proposition 13 and Proposition 14 in \cite{alkauskas}. For level $N\geq 2$, we need no additional calculations. Let us recall again the flows, given in \cite{alkauskas}, p. 321:
\begin{eqnarray}
\psi_{N}(\m{x})&=&\frac{(y+1)^{N}(x+y)+(x-y)}{(y+1)^{N}(x+y)-(x-y)}\cdot\frac{y}{y+1}\bl\frac{y}{y+1},\text{ and}\nonumber\\
\psi'_{N}(\m{x})&=&\frac{(y+1)^{N}(x-y)+(x+y)}{-(y+1)^{N}(x-y)+(x+y)}\cdot\frac{y}{y+1}\bl\frac{y}{y+1}.
\label{psi}
\end{eqnarray}
\begin{prop}
Let $N\geq 2$. Then the only $i$-symmetric level $N$ flows are given by
\begin{eqnarray*}
\phi(\m{x})=\ell^{-1}\circ\psi_{N}\circ\ell,\text{ or }\phi(\m{x})=\ell^{-1}\circ\psi'_{N}\circ\ell,
\end{eqnarray*}
where $\ell$ is given by $xA\bl yA$ for any symmetric, $0$-homogeneous $A$. 
\end{prop}
\begin{proof}Indeed, $\Phi_{N}$ and $\Phi'_{N}$, as given in \cite{alkauskas}, p. 322, are obtained from $\psi_{N}$ and $\psi'_{N}$ via conjugation with $xA\bl y A$, where $A(x,y)=\frac{x+y}{y}$. But this $1-$BIR is just equal to $\ell^{-1}$.
\end{proof}

For level $1$ flows, we have the following.
\begin{prop}There exist 4 basic $i$-symmetric level $1$ rational flows:
\begin{eqnarray*}
\phi_{1}(\m{x})&=&\frac{(2xy^2+x^2+2xy+y^2)x}{(xy+x+y)^2}
\bl\frac{(2xy^2+x^2+2xy+y^2)y}{(y^2+x+y)(xy+x+y)},\\
\psi_{1}(\m{x})&=&\frac{xy+y^2+2x}{(2+x+y)(y+1)}\bl\frac{y}{y+1},\\
\psi'_{1}(\m{x})&=&\frac{xy-y^2+2x}{(2-x+y)(y+1)}\bl\frac{y}{y+1},\\
\phi'_{1}(\m{x})&=&\frac{(2xy^2+x^2-2xy+y^2)x}{(xy-x+y)^2}
\bl\frac{(2xy^2+x^2-2xy+y^2)y}{(y^2+x-y)(-xy+x-y)}.
\end{eqnarray*}
All other $i$-symmetric level $1$ flows are obtained from these via conjugation with $1$-BIR $xA\bl yA$, where $A$ is $0$-homogeneous and symmetric. 
\end{prop}  
The notation $\psi_{N}$ and $\psi'_{N}$ is the same as in (\ref{psi}). For $ N=1$ it has the form given in the above Proposition. Note that the first coordinate of $\psi_{1}$ is equal to $\mathcal{U}_{\frac{1}{2},1}$ ,the and first coordinate of $\psi'_{1}$ is equal to $\mathcal{U}_{\frac{1}{2},-1}$ (see \cite{alkauskas}, p. 286).\\

One more interesting fact is that $\psi_{1}$ and $\psi'_{1}$ are linearly conjugate with respect to an involution $(x,y)\mapsto (-x,y)$. So are $\phi_{1}$ and $\phi'_{1}$ with respect to an involution $(x,y)\mapsto (x,-y)$. 

\section{Solenoidal rational flows}
\label{solenoidal}
\subsection{Motivation}
Two dimensional vector field $\varpi\bl\varrho$ is called \emph{solenoidal} if
\begin{eqnarray*}
\varpi_{x}+\varrho_{y}=0.
\end{eqnarray*}
This is also knwon as \emph{incompressible} or \emph{divergence-free} vector field. Such a flow preserves areas. In this section we will find all solenoidal rational vector fields in dimension $2$.
\begin{Example} The flow $\phi_{\mathrm{sph},\infty}(\m{x})=x+(x-y)^2\bl y+(x-y)^2$, as given by (\ref{sph}), is solenoidal, since its vector field is $\varpi\bl\varrho=(x-y)^2\bl(x-y)^2$. Let $\mathcal{A}$ be any compact set in $\mathbb{R}^{2}$ with a smooth boundary and connected complement. The fact that the flows $\phi_{\mathrm{sph},\infty}$ preserves areas follows easily from the formula
\begin{eqnarray*}
|\mathcal{A}|=\int\limits_{-\infty}^{\infty}{A_t}\d t,
\end{eqnarray*} 
where $A_{t}$ is the length of the intersection of the line $x-y=t$ with $\mathcal{A}$. This length does not change under the flow $\phi_{\mathrm{sph},\infty}$.
\label{example2} 
\end{Example}
\begin{Example}
Let $\phi_{3}$, as already defined by Theorem \ref{mthm}, be the canonic flow of level $3$, as given by (\ref{can}): $\phi_{3}(x,y)=x(y+1)^{2}\bl\frac{y}{y+1}$. Its vector field is given by $2xy\bl(-y^2)$, and so is also solenoidal. We will check that the area of the unit circle, or any other compact domain bounded by a smooth curve, is preserved under this flow. First,
\begin{eqnarray*}
u^{z}(x,y)\bl v^{z}(x,y)=x(zy+1)^{2}\bl\frac{y}{zy+1}:=A\bl B.
\end{eqnarray*} 
If $x^2+y^2=1$, then $(A,B)$ lies on the algebraic genus $0$ curve $\mathscr{I}_{z}$, given by
\begin{eqnarray*}
A^2(1-zB)^4+\frac{B^2}{(1-zB)^2}=1.
\end{eqnarray*}
Figure \ref{sole} shows these curves for $z=\frac{i}{7}$, $i=0,\ldots,4$.  
\begin{figure}
\epsfig{file=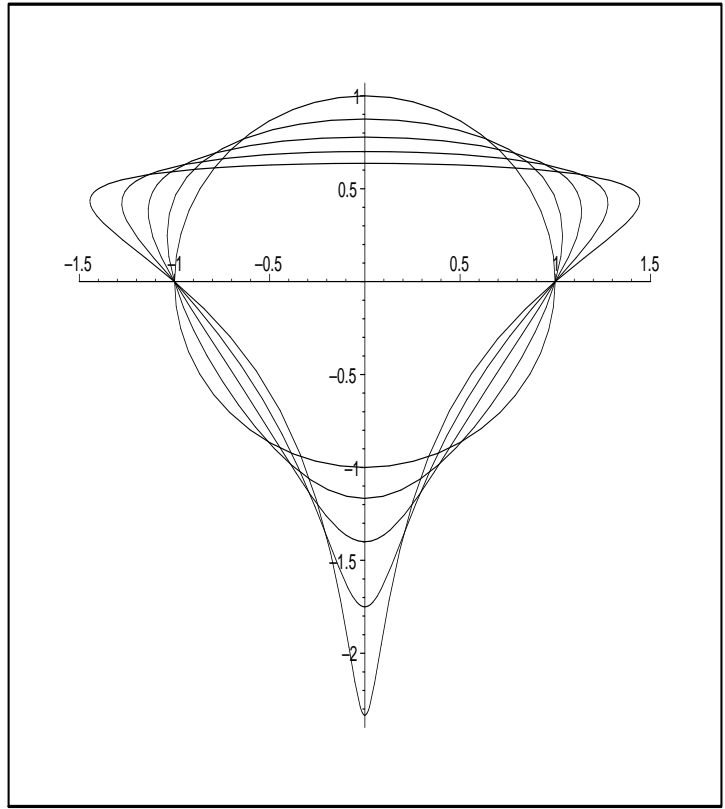,width=305pt,height=270pt,angle=-90}
\caption{The deformation of the unit circle under the solenoidal flow $\phi_{3}(\m{x})$}
\label{sole}
\end{figure}
Let $|z|<1$. We will show that the area inside $\mathscr{I}_{z}$ is always equal to $\pi$. In fact, let $\mathbf{X}(\theta)=(x(\theta),y(\theta))$, $\theta\in[0,2\pi]$, parametrizes any smooth curve $\mathscr{X}_{0}$, $\mathbf{X}(0)=\mathbf{X}(2\pi)$, $\mathbf{X}(\theta)\neq\mathbf{X}(\varphi)$ for $0\leq\theta<\varphi<2\pi$. Consider the curve $\mathscr{X}_{z}$, obtain from $\mathscr{X}_{0}$ under the flow $\phi_{3}$ after $z$ units of time. Thus, this curve is parametrized by
\begin{eqnarray*}
A(\theta)=x(\theta)\big{(}1+zy(\theta)\big{)}^2,\quad B(\theta)=\frac{y(\theta)}{1+zy(\theta)}.
\end{eqnarray*}
According to the Green's formula, the area inside $\mathscr{X}_{z}$ is equal to 
\begin{eqnarray*}
\oint\limits_{\mathscr{X}_{z}} A\d B=\int\limits_{0}^{2\pi} A(\theta)B'(\theta)\d \theta=\int\limits_{0}^{2\pi}x(\theta)y'(\theta)\d \theta=\oint\limits_{\mathscr{X}_{0}} x\d y.
\end{eqnarray*}
So, as expected, $\phi_{3}$ is an area-preserving flow.
\label{example3}
\end{Example}
As we will shortly see, these two examples is essentially the complete list (excluding the trivial flow $\phi_{\mathrm{id}}(x,y)=x\bl y$) of solenoidal rational flows.  
\subsection{All solenoidal flows}
Suppose, a level $N=0$ rational flow is solenoidal. Let its vector field is given by $xJ(x,y)\bl y J(x,y)$, where $J(x,y)$ is a non-zero rational $1$-homogeneous function. The property of solenoidality gives
\begin{eqnarray*}
J(x,y)+xJ_{x}(x,y)+J(x,y)+yJ_{y}(x,y)=3J(x,y)=0.
\end{eqnarray*} 
Thus, we get a contradiction, and there exist no level $0$ solenoidal flows.\\

 Assume $N\geq 1$. We will act similarly as we did in \cite{alkauskas}, Section 5.3. \\
 
 Let $\varpi(x,y)=Ux^2+Vxy+Wy^2=Q(x,y)-xy$, $\varrho(x,y)=-y^2$, $(V+1)^2-4UW=N^2$. Let this flow be $\phi$. If $Q=0$, this gives the vector field $(-xy)\bl(-y^2)$, this and all its conjugates are flows of level $0$, and we excluded this possibility. Therefore, assume $Q\neq 0$. Let $1$-BIR $\ell$ be given by $xA\bl yA$. The flow $\ell^{-1}\circ\phi\circ\ell$ has a vector field $\varpi'\bl\varrho'$, given by (\ref{vecconj}); that is,
\begin{eqnarray}
A(x,y)(Q(x,y)-xy)+yA_{y}(x,y)Q(x,y)&=&\varpi'(x,y),\nonumber\\
-A(x,y)y^2-yA_{x}(x,y)Q(x,y)&=&\varrho'(x,y).\label{nuo-cia}
\end{eqnarray}
Any vector field of a level $N$ rational flow can be obtained this way. We solve these equations for $A$ and $Q$ with the assumption that $\varpi'_{x}+\varrho'_{y}=0$. This gives 
\begin{eqnarray*}
A_{x}(Q-xy)+A(Q_{x}-y)+yA_{xy}Q+yA_{y}Q_{x}-A_{y}y^2-2Ay-A_{x}Q-yA_{xy}Q-yA_{x}Q_{y}=0.
\end{eqnarray*}
The mixed second derivatives vanish. Now, recall that $xQ_{x}+yQ_{y}=2Q$, and $xA_{x}+yA_{y}=0$. Thus, the above can be simplified as
\begin{eqnarray*}
A(Q_{x}-3y)-2A_{x}Q=0\Longrightarrow\frac{A_{x}}{A}=\frac{Q_{x}-3y}{2Q}.
\end{eqnarray*}
Now, $Q_{x}-3y=2Ux+(V-2)y$. We thus get the condition that
\begin{eqnarray}
A(x,1)=\exp\Bigg{(}\int\frac{2Ux+(V-2)}{2(Ux^2+(V+1)x+W)}\d x\Bigg{)}
\label{rat}
\end{eqnarray}
is a rational function.
\subsection{Level $N=1$} For level $N=1$, we know that the representation of level $1$ flows in an univariate form is not unique. In the simplest form, the first coordinate of the vector field is either $\tau x^{2}$, $\tau\in\mathbb{R}$, or $-2xy$. Suppose, it is the first case. So, in (\ref{rat}) we may suppose that $U=\tau$, and $V=W=0$. Then 
\begin{eqnarray*}
R(x)=\frac{2\tau x-2}{2\tau x^2+2x}=\frac{\tau x-1}{\tau x^2+x}=-\frac{1}{x}+\frac{2}{x+\frac{1}{\tau}}.
\end{eqnarray*} 
So,
\begin{eqnarray*}
\exp\Big{(}\int R(x)\d x\Big{)}=C\frac{(\tau x+1)^2}{x}\Longrightarrow A(x,y)=\frac{(\tau x+y)^2}{xy},
\end{eqnarray*}
and is always a rational function. Thus, we get a solenoidal flow of level $1$ by conjugating $\phi=\frac{x}{1-\tau x}\bl\frac{y}{y+1}$ with $1$-BIR $xA\bl y A$. We get a flow
\begin{eqnarray*}
x+(\tau x+y)^{2}\bl y-\tau(\tau x+y)^2.
\end{eqnarray*}
This flow is linearly conjugate to $\phi_{\mathrm{sph},\infty}$.
\\

Now, assume we have a vector field $-2xy\bl(-y^2)$, and so $U=W=0$, $V=-2$.
Then
\begin{eqnarray*}
\exp\Big{(}\int\frac{-4}{-2x}\d x\Big{)}=x^2.
\end{eqnarray*}
So, $A(x,y)=\frac{x^2}{y^2}$, and $\phi(\m{x})=\frac{x}{(y+1)^2}\bl\frac{y}{y+1}$. Conjugating $\phi$ with $xA\bl yA$, we get the solenoidal flow
\begin{eqnarray*}
x\bl x^{2}+y,\text{ with the vector field }0\bl x^2.
\end{eqnarray*}
This flow is also, however, linearly conjugate to $\phi_{\mathrm{sph},\infty}$. Thus we get the following claim.
\begin{prop}
The flow $\phi_{\mathrm{sph},\infty}=x+(x-y)^2\bl y+(x-y)^2$ as given by Example \ref{example2}, is the unique, up to linear conjugation, solenoidal rational projective flow of level $1$.
\end{prop}
\subsection{Level $N\geq 2$}Suppose now, $U=0$. Then
\begin{eqnarray*}
R(x)=\frac{2Ux+(V-2)}{2(Ux^2+(V+1)x+W)}=\frac{V-2}{2(V+1)x+2W}.
\end{eqnarray*}
So,
\begin{eqnarray}
\exp\Big{(}\int R(x)\d x\Big{)}=\Big{(}x+\frac{W}{V+1}\Big{)}^{\frac{V-2}{2(V+1)}}.
\label{a-show}
\end{eqnarray}
This is rational if $\frac{V-2}{2(V+1)}\in\mathbb{Z}$. But the property $(V+1)^2-4UW=(V+1)^2=N^2$ gives $V=\pm N-1$, and so
\begin{eqnarray}
\frac{N-3}{2N}\in\mathbb{Z},\text{ or }\frac{N+3}{2N}\in\mathbb{Z}.
\label{nn}
\end{eqnarray}
Since $N$ is positive, this gives $N=1$ (the case investigated in the previous subsection), or $N=3$. Suppose $N=3$, $V=2$. Then $R(x)=0$, $A(x,y)=1$, and we get the flow with the vector field $2xy+Wy^2\bl y^{2}$, exactly linearly conjugate to the flow in Example \ref{example3}, the canonical rational flow of level $3$. Suppose now $V=-4$. We have a flow $\phi$ with vector field $-4xy+Wy^2\bl (-y^2)$ of rational flow of level $3$. Next, (\ref{a-show}) shows that
\begin{eqnarray*}
A(x,y)=\frac{x}{y}-\frac{W}{3}.
\end{eqnarray*}
Conjugating $\phi$ with $1$-BIR of the form $xA\bl yA$, we according to Proposition \ref{prop1}, get a vector field
\begin{eqnarray*}
-x^2+\frac{4}{3}Wxy-\frac{1}{3}W^2y^2\bl2xy-\frac{2}{3}Wy^2.
\end{eqnarray*}
This is a solenoidal vector field, linearly conjugate to $2xy\bl(-y^2)$. Indeed, conjugating this with  a linear  map $(x,y)\mapsto (x,3y/W)$, we get a vector field
\begin{eqnarray*}
-x^2+4xy-3y^2\bl 2xy-2y^2=P\bl Q.
\end{eqnarray*}
Now, $yP-xQ=-3y(x-y)^2$. We will now use the formulas from \cite{alkauskas}, p. 308, which allow to perform a linear change and transform a pair of quadratic forms into the simplest form. Thus, let us consider the linear change $(x,y)\mapsto (x,x+y)$. This gives the vector field $-x^2\bl 2xy$, and after another conjugation with $(x,y)\mapsto (y,x)$, this again yields the flow $\phi_{3}$.\\

Now, assume $U\neq 0$. The function on the right of (\ref{rat}) is rational if and only if
\begin{eqnarray*}
R(x)=\frac{2Ux+(V-2)}{2(Ux^2+(V+1)x+W)}=\frac{a}{x-\xi_{1}}+\frac{b}{x-\xi_{2}},\quad a,b,\in\mathbb{Z};
\end{eqnarray*} 
here $\xi_{1}\xi_{2}=\frac{W}{U}$, $\xi_{1}+\xi_{2}=-\frac{V+1}{U}$. This gives
\begin{eqnarray*}
2Ua(x-\xi_{2})+2Ub(x-\xi_{1})=2Ux+(V-2).
\end{eqnarray*}
This, on its turn, gives two linear equations for $a$ and $b$. Solving we obtain
\begin{eqnarray*}
a=\frac{2U\xi_{1}-2+V}{2U(\xi_{1}-\xi_{2})},\quad b=\frac{2-V-2U\xi_{2}}{2U(\xi_{1}-\xi_{2})}.
\end{eqnarray*}
Since $a+b=1$, this gives the single condition
\begin{eqnarray*}
\frac{2U\xi_{1}-2+V}{2U(\xi_{1}-\xi_{2})}\in\mathbb{Z}.
\end{eqnarray*}
But we have
\begin{eqnarray*}
a=\frac{2U\xi_{1}-2+V}{2U(\xi_{1}-\xi_{2})}=\frac{2U\xi_{1}-3-U(\xi_{1}+\xi_{2})}{2U(\xi_{1}-\xi_{2})}=\frac{1}{2}-\frac{3}{2U(\xi_{1}-\xi_{2})}=M\in\mathbb{Z}.
\end{eqnarray*}
Further,
\begin{eqnarray*}
U(\xi_{1}-\xi_{2})=\sqrt{(V+1)^2-4UW}=\pm N.
\end{eqnarray*}
So,
\begin{eqnarray*}
\frac{1}{2}\pm\frac{3}{2N}\in\mathbb{Z}.
\end{eqnarray*}
Note that this is completely identical to (\ref{nn}). Since $N\geq 2$, this again gives $N=3$, $(a,b)=(1,0)$ or $(0,1)$. In both cases, we have
\begin{eqnarray*}
R(x)=\frac{1}{x+\frac{V+1+N}{2U}}=\frac{1}{x+\frac{V+4}{2U}}.
\end{eqnarray*}
Consider, thus, a vector field $Ux^2+Vxy+Wy^2\bl(-y^2)$, and let us again calculate $\varpi'\bl\varrho'$ for $A=\frac{x}{y}+\frac{V+4}{2U}$. We verify once again that $\varpi'\bl\varrho'$ is a pair of quadratic forms, linearly conjugate to $2xy\bl (-y^2)$.  Thus, we have proved
\begin{prop}For level $N\geq 2$, solenoidal rational flows exist only for level $N=3$. For level $N=3$, the flow $\phi_{3}(\m{x})=x(y+1)^2\bl\frac{y}{y+1}$ as given by Example \ref{example3}, is the unique, up to linear conjugation, solenoidal rational projective flow of level $3$.
\end{prop} 

\subsection{Higher dimensional rational solenoidal flows}The flow 
$\phi_{\mathrm{sph},\infty}(\m{x})$ can be generalized to any dimension. In fact, this generalization was given in our first paper on the projective translation equation \cite{alkauskas-0}.\\

Let $N\in\mathbb{N}$, $N\geq 2$. Let us choose a non-zero linear form $L(\m{x})=\sum_{i=1}^{N}a_{i}x_{i}$ in $N$ variables, and let $\m{c}=(c_{1},\ldots,c_{N})\in\mathbb{R}^{N}$ be such that $\m{c}\neq\m{0}$, $L(\m{c})=0$. Then
\begin{eqnarray*}
\phi_{\m{c},L}(\m{x})=\m{c}L^{2}(\m{x})+\m{x}
\end{eqnarray*}  
is a rational projective flow with a vector field $\varpi_{1}\bl\cdots\bl\varpi_{N}=\m{c}L^{2}(\m{x})=c_{1}L^{2}(\m{x})\bl\cdots\bl c_{N}L^{2}(\m{x})$. It is solenoidal, since
\begin{eqnarray*}
\sum\limits_{i=1}^{N}\frac{\p}{\p x_{i}}\varpi_{i}=2L(\m{x})\sum\limits_{i=1}^{N}a_{i}c_{i}=0.
\end{eqnarray*}
This facts nicely complements the theory of superflows, since all discovered real superflows in dimension $N\geq 3$ turn out to be solenoidal \cite{alkauskas-super2}.\\

Further, let $n,m\in\mathbb{Z}$, and consider rational flow in dimension $N=3$, given in \cite{alkauskas}, p. 324:
\begin{eqnarray*}
\psi_{n,m}(x,y,w)=x(w+1)^{n-1}\bl y(w+1)^{m-1}\bl\frac{w}{w+1}.
\end{eqnarray*} 
Its vector field is equal to $(n-1)xw\bl(m-1)yw\bl(-w^2)$, and so is solenoidal if and only if $n+m=4$.\\

 In particular, the rational flow  
$\psi(\m{x})=x(w+1)\bl y(w+1)\bl\frac{w}{w+1}$ preserves volumes. Consider the unit sphere $x^2+y^2+w^2=1$. After $z$ units of time, the point $(x,y,w)$ under the flow $\psi$ is carried into the point $x(zw+1)\bl y(zw+1)\bl\frac{w}{zw+1}$, and so the unit sphere - into the genus $0$ surface
\begin{eqnarray*}
A^2(1-zC)^2+B^2(1-zC)^2+\frac{C^2}{(1-zC)^2}=1.
\end{eqnarray*}
Similarly as in Example \ref{example3}, let 
\begin{eqnarray*}
(x^{0},y^{0},z^{0})=\big{(}x(\theta,\varphi),y(\theta,\varphi),w(\theta,\varphi)\big{)}, 
\quad (\theta,\varphi)\in\Delta,
\end{eqnarray*}
is a parametrization of a smooth, compact, oriented, closed and not self-intersecting surface $\mathcal{S}_{0}$. After $z$ units of time, this surface is carried onto a surface, parametrized by $(x^{z},y^{z},w^{z})=\big{(}x(1+zw),y(1+zw),\frac{w}{1+zw}\big{)}$. According to Gauss' formula, the volume inside this surface is equal to
\begin{eqnarray*}
V_{z}=\iint\limits_{\mathcal{S}_{z}}x^{z}\d y^{z}\d w^{z}.
\end{eqnarray*}
But
\begin{eqnarray*}
x_{z}\frac{D(y^{z},w^{z})}{D(\theta,\varphi)}=x(1+zw)\begin{vmatrix}
y_{\theta}(1+zw)+zyw_{\theta}   &\frac{w_{\theta}}{(1+zw)^2}\\
y_{\varphi}(1+zw)+zyw_{\varphi} &\frac{w_{\varphi}}{(1+zw)^2}\end{vmatrix}=
x\begin{vmatrix}
y_{\theta} &w_{\theta}\\
y_{\varphi}&w_{\varphi}
\end{vmatrix}.
\end{eqnarray*} 
So, as expected, $V_{z}=V_{0}$.\\
 
Note that flows $\psi_{n,m}$ and $\psi_{m,n}$ are linearly conjugate.\\

One more solenoidal flow is obtained from Theorem \ref{thm2}. Let now $x,y,z$ be three space variables, and let $\phi=\frac{x}{x+1}\bl\frac{y}{y+1}$, $\mathscr{W}=zx^Ay^{B}$, $A,B\in\mathbb{Z}$. This gives a rational $3$-dimensional flow
\begin{eqnarray*}
\Phi_{A,B}=\frac{x}{x+1}\bl\frac{y}{y+1}\bl z(x+1)^{A}(y+1)^B.
\end{eqnarray*}
Consider the involution $\ell(\m{x})=x\bl y\bl \frac{x^2}{z}$. Then, by a direct calculation, $\ell^{-1}\circ\Phi_{A,B}\circ\ell=\Phi_{2-A,-B}$. Thus, flows with pairs $(A,B)$ and $(-2-A,-B)$ are $1$-BIR conjugate, as well as flows with pairs $(A,B)$ and $(-A,-2-B)$. Consider an involution $\ell(\m{x})=x\bl y\bl\frac{xy}{z}$. Tis gives $1$-BIR conjugate pairs $(A,B)$ and $(-1-A,-1-B)$. Combining $(-2-A,-B)$ and $(-1-A,-1-B)$, gives $(A+1,B-1)$. The iteration of this gives a pair $(A+B,0)$.\\

A flow $\Phi_{A,B}$ is solenoidal only if $A=B=2$, and thus we have another example
\begin{eqnarray*}
\frac{x}{x+1}\bl\frac{y}{y+1}\bl z(x+1)^2(y+1)^2.
\end{eqnarray*}
  Continuing the list of open problems related to the projective translation equation begun in \cite{alkauskas-un} and continued in \cite{alkauskas-comm, alkauskas-super1, alkauskas-super2}, we pose one more problem. 
\begin{prob14}
Find all $3$-dimensional rational projective and solenoidal flows.
\end{prob14}
Among the flows $\psi_{n,4-n}$ there are some which are $\ell$-conjugate, but not linearly conjugate. For example, as is clear from (\cite{alkauskas}, Subsection 5.6), rational flows $\psi_{3,1}$ and $\psi_{5,-1}$ are both $1$-BIR conjugate to $\psi_{1,0}$, but they are not linearly conjugate. This corresponds to the fact that Cremona group of $P(\mathbb{C}^{2})$ contains not only M\"{o}bius transformations, but also quadratic maps. \\

Also, note that in the light or results of the current paper, the task to classify all rational projective flows in dimension $3$ seems much more optimistic than it seemed before. We do not need the algorithm of reduction of vector fields; some analogue of Theorem \ref{thm1} should hold, and we can obtain the existence of $1$-BIR, which carries rational flow into a canonic form,  by some ``bootstrapping" argument. 

\section{Rational flows which share orbits}
\label{orbits-share} 
 In (\cite{alkauskas-ab}, Subection 4.3) we have seen that two algebraic flows
\begin{eqnarray*}
\phi^{(2)}_{3}(x,y)&=&\frac{(y^2+x)^3}{(x+2xy+y^3)^2}\bl\frac{y(y^2+x)}{(x+2xy+y^3)},\\
\tilde{\varpi}\bl\tilde{\varrho}&=&(-4xy+3y^2)\bl(-2xy^2+y^3)x^{-1},
\end{eqnarray*}
and
\begin{eqnarray*}
\Phi(\m{x})&=&\frac{y\sqrt{4x+(y-1)^2}+y^2+2x-y}{8x+2(y-1)^2}\bl \frac{y}{\sqrt{4x+(y-1)^2}},\\ 
\varpi\bl\varrho&=&(-4x^2+3xy)\bl(-2xy+y^2),
\end{eqnarray*}
share the same orbits $\mathscr{W}=\frac{y^4}{x(x-y)}=\mathrm{const}.$, since
\begin{eqnarray*}
\frac{\varpi(x,y)}{\varrho(x,y)}=\frac{\tilde{\varpi}(x,y)}{\tilde{\varrho}(x,y)}.
\end{eqnarray*}
In this Section we find when this happens for rational flows. Because of the definition of the level, only flows of the same level can share orbits.\\

It is obvious, that if $\ell$ is a $1$-BIR, and rational flows $\phi$ and $\psi$ share the orbits, so do $\ell^{-1}\circ\phi\circ\ell$ and $\ell^{-1}\circ\psi\circ\ell$. This is clear from Proposition \ref{prop1} - vector field $\varpi\bl\varrho$ and $\alpha\bl\beta$ are proportional (by some $0$-homogeneous function, which in the above case is $\frac{x}{y}$), so are $\varpi'\bl\varrho'$ and $\alpha'\bl\beta'$. And, of course, linear conjugation preserves the property of coinciding orbits.  

\subsection{Level $0$} Since for any rational $1$-homogeneous function $J(x,y)$ the flow of level $0$ $\phi=\frac{x}{1-J(x,y)}\bl\frac{y}{1-J(x,y)}$ has orbits $\frac{x}{y}=\mathrm{const.}$, all such flows share the same orbits.  
\subsection{Level $1$}Let two flows $\phi$ and $\psi$ of level $1$ share orbits. This level is exceptional because there exists a $1$-BIR such that a flow $\ell^{-1}\circ\phi\circ\ell$ has a vector field with the second coordinate vanishing, so the orbits are given by $y=\mathrm{const}$. So, the orbits of  $\ell^{-1}\circ\psi\circ\ell$ are given by the same equation, and so its second coordinate of the vector field also vanish. Thus, we need to find all  flows of level $1$ with a vector field of the form $\varpi\bl 0$.\\

This was done in Subsection 4.2 in \cite{alkauskas}, only one tiny mistake is left in the formulation of Proposition 10. Since we (see \cite{alkauskas}, p. 307) use conjugation with a linear map $L(x,y)=(x+py,y)$, the correct formulation should read as follows:
\begin{prop}Assume that $\phi(\m{x})$ is a rational real flow, and that $\varrho(x,y)\equiv 0$. Then $\varpi(x,y)=(ax+by)^2$ for certain $a,b\in\mathbb{R}$.
\end{prop}
Thus, over reals, up to conjugation with a homothety, these are the vector fields: $\varpi_{a}(x,y)=(x+ay)^2\bl 0$, and $y^2\bl 0$. And so, these are the flows:
\begin{eqnarray*}
\phi_{a}(\m{x})=\frac{x+ayx+a^2y^2}{1-x-ay}\bl y,\quad a\in\mathbb{R},\text{ and }\psi_{0}(\m{x})=x+y^2\bl y.
\end{eqnarray*} 
Thus, we have proved
\begin{prop}Suppose, rational flows $\phi$ and $\psi$ of level $1$ share orbits. Then there exists a $1$-BIR $\ell$ such that two flows $\ell^{-1}\circ\phi\circ\ell$ and $\ell^{-1}\circ\psi\circ\ell$, belong, up to conjugation with a homothety, to the set $\{\phi_{a},a\in\mathbb{R}\}\cup\{\psi_{0}\}$.
\end{prop}
By ``conjugation with a homothety" we mean, as always, a flow $z^{-1}\phi(z\m{x})$.   
\subsection{Level $N\geq 2$} Let $\mathscr{W}$ be equations for the orbit of a rational flow of level $N\geq 2$. We know (\cite{alkauskas}, Corollary 1) that there exists two homogeneous polynomials $P$ and $Q$ of the same degree and a non-degenerate linear map $L$ such that
\begin{eqnarray}
\mathscr{W}(x,y)=\frac{P^{N}(x,y)}{Q^{N}(x,y)}xy^{N-1}\circ L(x,y).
\label{eq-orbits}
\end{eqnarray}
Let, again, $\phi$ and $\psi$ be two flows that share the same orbits. We can choose $1$-BIR $\ell$ such that $\ell^{-1}\circ\phi\circ\ell$ is a flow in a canonic form (\ref{can}). Let the orbits of $\ell^{-1}\circ\phi\circ\ell$ be then given by (\ref{eq-orbits}). Thus, we have
\begin{eqnarray*}
xy^{N-1}=\frac{P^{N}}{Q^{N}}TL^{N-1}\Longrightarrow \frac{xy^{N-1}}{TL^{N-1}}=\frac{P^{N}}{Q^{N}}.
\end{eqnarray*} 
where $T,L$ are certain linear forms, $\frac{T}{L}\neq\mathrm{const.}$ There are two cases two consider.\\

 If $N\geq 3$, then the above implies that $T=cx$, $L=dy$, $P,Q$ are constants, and the flows, up to conjuagtion with a diagonal linear map, coincide.\\

If $N=2$, we can also have $T=cy$, $L=dx$, $P,Q$ are constants, and the flow $\ell^{-1}\circ\psi\circ\ell$ is obtained from $\ell^{-1}\circ\phi\circ\ell$  by conjugation with an involution $i_{0}(x,y)=(y,x)$. Thus, we have proved

\begin{prop}If flows $\phi$ and $\psi$ of level $N \geq 3$ have the same orbits, then there exists a $1$-BIR $\ell$ such that both $\ell^{-1}\circ\phi\circ\ell$ and $\ell^{-1}\circ\psi\circ\ell$ are flows of the form
\begin{eqnarray*}
\phi^{c}_{N}(\m{x})=x(cy+1)^{N-1}\bl\frac{y}{cy+1},\quad c\in\mathbb{R}\setminus\{0\}.
\end{eqnarray*}
  If flows of level $2$ have the same orbits, there exists a $1$-BIR $\ell$ such that both $\ell^{-1}\circ\phi\circ\ell$  and $\ell^{-1}\circ\psi\circ\ell$ are flows of the form
\begin{eqnarray*}
\text{either }\phi^{c}_{2}(\m{x})=x(cy+1)\bl\frac{y}{cy+1}\quad\text{or }\widehat{\phi^{d}_{2}}(\m{x})=\frac{x}{dx+1}\bl y(dx+1),\quad c,d\in\mathbb{R}\setminus\{0\}.
\end{eqnarray*}   
\end{prop}
The flows $\phi^{c}_{2}$ and $\widehat{\phi^{d}_{2}}$ both have hyperbolas $xy=\mathrm{const}.$ as their orbits.
\begin{Example} Let $\ell(x,y)=i(x,y)=\frac{y^2}{x}\bl y$. The flows $i\circ\phi^{2}_{N}\circ i$, $i\circ\phi_{N}\circ i$ and $i\circ\widehat{\phi_{N}}\circ i$ are given by
\begin{eqnarray*}
\frac{x}{(y+1)^{3}}\bl\frac{y}{y+1},\quad \frac{x}{(2y+1)^{3}}\bl\frac{y}{2y+1},\quad \frac{(y^2+x)^{3}}{x^2}\bl \frac{y(y^2+x)}{x}
\end{eqnarray*}
(the last flow is denoted by $\phi^{(2)}_{1}(\m{x})$ in \cite{alkauskas}, see p. 288 and Table 1), and all have orbits $\frac{y^3}{x}=\mathrm{const}$.
\end{Example}
Example given in the beginning of this subsection shows that there are much more interesting phenomena among algebraic flows. This question, also question of solenoidal, symmetric algebraic flows, among others,  will be treated in \cite{alkauskas-ab2}. \\

In a similar analysis we could give all rational plane flows with orthogonal orbits. We do not give the complete solution, but rather confine to two examples. 

\begin{Example}
Two rational flows of level $1$, with vector fields $\varpi\bl 0$ and $0\bl\varrho$ have, of course, orthogonal orbits. For example, such two flows are $x+y^2\bl y$ and $x\bl\frac{y}{y+1}$.
\end{Example}
\begin{Example} Consider the flow (see \cite{alkauskas-0}, also \cite{alkauskas}, Subsection 2.3, and further Table 1 on p. 290)
\begin{eqnarray*}
\phi_{\mathrm{sph},1}=\frac{x^2+y^2+2x}{(x+1)^2+(y+1)^2}\bl\frac{x^2+y^2+2y}{(x+1)^2+(y+1)^2} 
\end{eqnarray*}
with the vector field $\varpi\bl\varrho=-\frac{1}{2}x^2+\frac{1}{2}y^2-xy\bl\frac{1}{2}x^2-\frac{1}{2}y^2-xy$, and orbits $\frac{x^2+y^2}{x-y}=\mathrm{const}.$. These are shown in \cite{alkauskas}, Figure 1. These orbits are just circles $(x-a)^2+(y+a)^2=2a^2$.\\

Consider the vector field $\alpha\bl\beta=\frac{1}{2}x^2-\frac{1}{2}y^2-xy\bl \frac{1}{2}x^2-\frac{1}{2}y^2+xy$. Let $i_{1}$ be an involution $i_{1}(x,y)=(-x,y)$. Since $\alpha\bl\beta=i_{1}\circ(\varpi\bl\varrho)\circ i_{1}$, the integral of $\alpha\bl\beta$ is the flow
\begin{eqnarray*}
\frac{-x^2-y^2+2x}{(x-1)^2+(y+1)^2}\bl\frac{x^2+y^2+2y}{(x-1)^2+(y+1)^2} 
\end{eqnarray*}
with orbits $\frac{x^2+y^2}{x+y}=\mathrm{const}.$; these are just circles $(x-a)^2+(y-a)^2=2a^2$. Thus, these two flows have mutually orthogonal orbits.
\end{Example}

\end{document}